\documentclass{article}
\pdfoutput=1
\usepackage[utf8]{inputenc}
\usepackage{ulem}
\usepackage{amsmath}
\usepackage{amsthm}
\usepackage{amssymb}
\usepackage{tikz}
\usepackage{graphicx}
\usepackage{float}
\usepackage{comment}
\usepackage{subfigure}
\graphicspath{ {./images/} }
\theoremstyle{definition}

\newtheorem{con}{Conjecture}[section]

\newtheorem{lem}{Lemma}[section]
\newtheorem{proposition}{Proposition}[section]
\newtheorem{col}{Corollary}[section]
\newtheorem{theo}{Theorem}[section]

\newtheorem*{ske}{Proof}

\usepackage{titlesec}

\titleformat{\paragraph}
{\normalfont\normalsize\bfseries}{\theparagraph}{1em}{}
\titlespacing*{\paragraph}
{0pt}{3.25ex plus 1ex minus .2ex}{1.5ex plus .2ex}
\DeclareMathOperator{\interior}{Int}
\usepackage[top=2.5cm, left=3cm, right=3cm, bottom=4.0cm]{geometry}
\author{Joel Foisy, Luis Ángel Topete Galván, Evan Knowles, Uriel Alejandro Nolasco,
\and Yuanyuan Shen, and Lucy Wickham}
\title{Links in projective planar graphs}

\begin{document}
\maketitle

\begin{abstract}
A graph $G$ is \textit{nonseparating projective planar} if $G$ has a projective planar embedding without a nonsplit link. Nonseparating projective planar graphs are closed under taking minors and are a superclass of projective outerplanar graphs. We partially characterize the minor-minimal separating projective planar graphs by proving that given a minor-minimal nonouter-projective-planar graph $G$, either $G$ is minor-minimal separating projective planar or $G \dot\cup K_{1}$ is minor-minimal weakly separating projective planar, a necessary condition for $G$ to be separating projective planar. 

One way to generalize separating projective planar graphs is to consider \textit{type I 3-links} consisting of two cycles and a pair of vertices. A graph is \textit{intrinsically projective planar type I 3-linked (IPPI3L)} if its every projective planar embedding contains a nonsplit type I 3-link. We partially characterize minor-minimal IPPI3L graphs by classifying all minor-minimal IPPI3L graphs with three or more components, and finding many others with fewer components.\\

\end{abstract}

\section{Introduction}

We say a projective planar graph $G$ is \textit{separating projective planar} if for every embedding of $G$ into $\mathbb{R}P^2$, there is a disk bounding cycle $C$ in $G$ such that one vertex of $G$ is in one connected component of ${\mathbb{R}P}^2\setminus C$ and another vertex of $G$ is in the other connected component of $\mathbb{R}P^2\setminus C$. If there exists an embedding of $G$ that does not have this property, $G$ is \textit{nonseparating projective planar}. We say a projective planar graph $G$ is \textit{strongly nonseparating} if there exists an embedding of $G$ into $\mathbb{R}P^2$, such that for every pair $(u, v)$ of vertices of $G$, there is a path in $\mathbb{R}P^2$ from $u$ to $v$ that intersects $G$ only at its endpoints. If a graph is not strongly nonseparating projective planar, this means it is \textit{weakly separating}. If a graph $G$ is separating, this implies it is weakly separating.

Dehkordi and Farr characterized the set of nonseparating planar graphs, identifying the forbidden minors for nonseparating planar graphs. They proved that a graph $G$ is a nonseparating planar graph if and only if it does not contain any of $K_1 \dot\cup K_4$ or $K_1 \dot\cup K_{2,3}$ or $K_{1,1,3}$ as a minor \cite{Dehkordi}. Their work showed that separating graphs are built from nonouterplanar graphs. Our work extends these concepts from the plane to the projective plane. There are 32 minor-minimal nonouter-projective-planar graphs \cite{Archdeacon}. One of the main results of this paper is the classification of these graphs as separating or nonseparating. We show every such graph is either separating, or its disjoint union with $K_1$ is weakly separating, which is a necessary condition for separating projective planar graphs. A nonplanar and nonouter-projective-planar graph disjoint union with $K_1$ is also separating projective planar. The set of minor-minimal separating projective planar graphs must be finite by Robertson and Seymour's result on graph minors \cite{robertson}. Though we have not yet characterized the complete set of minor-minimal separating projective planar graphs, we have identified many members. These graphs can be used to relate separating projective planar graphs and links in graphs embedded in projective space \cite{REU2021IPL}.\\

Define $S^k$ to be the $k$-sphere. A \textit{projective planar 3-link} is a disjoint collection of $3-m$ 1-spheres and $m$  0-spheres, embedded into the projective plane, where $m\in\{1,2\}$. If there are two $S^1$'s, this is a \textit{type I 3-link}. A graph $G$ is \textit{intrinsically projective planar type I 3-linked (IPPI3L)} if every embedding of $G$ in the projective plane has a nonsplit type I 3-link. If there are two $S^0$'s, this is a \textit{type II 3-link}. A graph $G$ is \textit{intrinsically projective planar type II 3-linked (IPPII3L)} if every embedding of $G$ in the projective plane has a nonsplit type II 3-link.

Burkhart et al examined 3-links in graphs embedded in $S^2$. They proved that the graphs $K_4 \dot\cup K_4$, $K_4 \dot\cup K_{3,2}$, and $K_{3,2} \dot\cup K_{3,2}$ are  minor-minimal with respect to being intrinsically type I 3-linked, and conjectured it forms a complete set \cite{Burkhart} of such minor-minimal graphs. Our research built on this foundation, identifying many types of projective planar graphs that are minor-minimal IPPI3L. We have proven there are no minor-minimal IPPI3L graphs that have four or more components, but there is a set with three components that is the disjoint union of $K_4$ and $K_{3,2}$ components. If a graph $G$ is \textit{closed nonseparating}, that means any nonseparating embedding of $G$ in the projective plane is a \textit{closed cell embedding}. If a drawing of graph $G$ in the projective plane is a \textit{closed cell embedding}, that means every face of the graph can be bounded by a 0-homologous cycle. We proved the disjoint union of two closed nonseparating graphs is also minor-minimal IPPI3L. Additionally, we prove separating or closed nonseparating graphs that are glued at a vertex can be IPPI3L under specified conditions. Finally, we prove graphs that are the disjoint union of two components, where one component is a graph that is nonplanar and nonouter-projective-planar, and the second component is a graph that is nonouterplanar, are IPPI3L, and can be minor-minimal in that regard under certain conditions.\\

\section{Definitions and notation}
Let $\mathbb{R}P^2$ denote the real projective plane. We represent $\mathbb{R}P^2$ as the unit disk in $\mathbb{R}^2$ with antipodal points identified. Let $\mathbb{R}P^3$ denote real projective space, which can be defined as the unit ball in $\mathbb{R}^3$ with antipodal points identified.

All of our graphs will be embedded piecewise linearly. Thus, for every graph embedded in ${\mathbb R}P^2$, we may assume every cycle intersects the boundary at a finite number of points. A cycle embedded in $\mathbb{R}P^2$ or $\mathbb{R}P^3$ is \textit{0-homologous} if and only if it bounds a disk. This is also called a \textit{null cycle} \cite{Glover}. A cycle embedded in $\mathbb{R}P^2$ or $\mathbb{R}P^3$ is \textit{1-homologous} if and only if it does not bound a disk. This is also called an \textit{essential cycle} \cite{Glover}. An embedding of graph in $\mathbb{R}P^2$ or $\mathbb{R}P^3$ is an \textit{affine embedding} if the graph embedding does not intersect the boundary of the ball used to define the projective plane.

A graph that can be obtained from a graph $G$ by a series of edge deletions, vertex deletions and edge contractions is called a \textit{minor} of $G$. The graph $G$ is \textit{minor-minimal} if, whenever $G$ has property $P$ and $H$ is a minor of $G$, then $H$ does not have property $P$. A property $P$ is \textit{minor closed} if, whenever a graph $G$ has property $P$ and $H$ is a minor of $G$, then $H$ also has property $P$. If $P$ is a minor closed property and the graph $G$ does not have property $P$, then $G$ is a \textit{forbidden graph} for $P$. Robertson and Seymour's Minor Theorem states if $P$ is a minor-closed graph property, then the minor-minimal forbidden graphs for $P$ form a finite set \cite{robertson}. 

An \textit{outer-projective-planar graph} is one that can be embedded in the projective plane with all vertices in the same face - this is a minor closed property. Other minor closed properties include having a nonseparating projective planar planar embedding and having a type I 3-linkless embedding in the projective plane.

A \textit{complete graph} is a graph with an edge between all possible pairs of vertices in the graph. We represent a complete graph as $K_m$, where $m$ is the order of the graph. A \textit{complete bipartite graph} is a graph whose vertices can be divided into two disjoint sets, where no two vertices in the same set are adjacent and every vertex of the first set is adjacent to every vertex of the second set. We represent a complete bipartite graph as $K_{n,m}$, where the cardinalities of the two sets are $n$ and $m$. A \textit{complete $k$-partite graph} is a graph whose vertices can be divided into $k$ disjoint sets, where no two vertices in the same set are adjacent and every vertex of a given set is adjacent to every vertex in every other set. We represent a complete $k$-partite graph as $K_{n_1, \dots, n_k}$, where the cardinalities of the sets are $n_1, \dots, n_k$.

\section{Results on cycles in the projective plane}

In this section, we prove some results about the characteristics of 0-homologous and 1-homologous cycles in the projective plane. These theorems are utilized in our proofs in the following sections. They lay the foundation for understanding how cycles function and interact in the projective plane.\\

The following result is well-know in graph theory. See for example, Diestel's text \cite{Diestel}.

\begin{lem}\label{lemma1}
A graph is 2-connected if and only if it can be constructed from a cycle by successively adding H-paths to graphs H already constructed.
\end{lem}

\begin{theo}\label{Theorem1}
Let G be a 2-connected planar graph embedded in the projective plane with all cycles 0-homologous. Then the embedding of $G$ can be isotoped to an affine embedding.
\end{theo}

\begin{proof}
Consider an arbitrary embedding of $G$ in the projective plane, with all its cycles 0-homologous. Pick the `first' cycle as in Lemma \ref{lemma1}. Since this cycle is 0-homologous and bounds a disk $D$, it can be isotoped (along with the rest of $G$, which may not all be affine yet) to an affine embedding with no point arbitrarily close to the boundary. 

Consider adding a H-path $H_{1}$ to the cycle. It intersects the cycle at two vertices $v_{1}$ and $v_{2}$, which divide the cycle into two paths $P_{1}$ and $P_{2}$. Let $L_{1}$ be the cycle consisting of $H_{1}$ and $P_{1}$, and $L_{2}$ of $H_{1}$ and $P_{2}$. By assumption both $L_{1}$, $L_{2}$ are 0-homologous and let $L_{1}$ bound $D_{1}$, $L_{2}$ bound $D_{2}$. Then $D$ is contained in exactly one of $D_{1}$, $D_{2}$. Without loss of generality assume $D \subseteq D_{1}$, $\interior D \cap \interior D_{2} = \varnothing$, $D \cup D_{2} = D_{1}$.

\begin{figure}[H]
\begin{center}
    \includegraphics[scale=0.5]{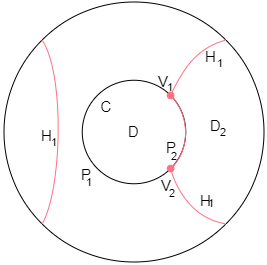}
    \caption{Adding a H-path $H_{1}$ to the cycle}
\end{center}
\end{figure}
 
Isotope $L_{2}$ (and the rest of $G$) to an affine embedding by deforming $D_{2}$ towards $D$, then every cycle in this embedding is affine. Isotope $L_{1}$ so that no point is arbitrary close to the boundary. Inductively add H-paths until the original 2-connected graph G is obtained in an affine embedding, equivalent to the original embedding. 
\end{proof}

Consider the following two definitions: the \textit{standard affine cycle} is the affine 0-homologous cycle $\{(x,y) | x^2 + y^2 = 1/2\}$ and the \textit{standard vertical cycle} is the 1-homologous cycle $\{(x,y) | x=0\}$ intersecting the boundary at exactly one point.

In order to prove the second theorem, that every 0-homologous cycle is isotopic to the standard affine cycle and every 1-homologous cycle is isotopic to the standard vertical cycle in the projective plane, the following lemma is needed. 

\begin{lem}
Given a cycle in $\mathbb{R}P^2$ crossing the boundary transversely at exactly $p_{1}$, $p_{2}$, ... $p_{2k}$ labeled clockwise, where $p_{i}$, $p_{j}$ represent the same point if and only if $i \equiv j \mod k$, then there exists an $i$ where $1\leq i \leq 2k-1$ such that there is a path that is a subgraph of the cycle, connecting $p_{i}$ and $p_{i+1}$ and that intersects the boundary only at $p_{i}$ and $p_{i+1}$.
\end{lem}
\begin{proof}
Consider the path connecting $p_{1}$ to some $p_{n_{1}}$, $2 \leq n_{1} \leq 2k$, which intersects the boundary at only $p_{1}$ and $p_{n_{1}}$. We call such a path an \textit{interior path}. If $n=2$, this is the desired path. Else, $p_{2}$ must connect to some $p_{n_{2}}$, $3 \leq n_{2} \leq n_{1}-1$ via an interior path that intersects the boundary only at $p_{2}$, $p_{n_{2}}$. Inductively consider paths connecting $p_{i}$ with $p_{n_{i}}$, $1 \leq i \leq k$. 

Set $m_{i}=n_{i}-i$, then $m_{i}$ is positive and strictly decreasing. Then $m_{a} = 1$ or $m_{a} = 2$ for some $a$: if $m_{b} \geq 3$ for all $b$, then for some $b$ there is an interior path connecting $p_{b+1}$ with $p_{n_{b+1}}$ such that $m_{b+1} < m_{b}$.

If $m_{a} = 1$, then $p_{a}$ with $p_{n_{a}}$ is the desired interior path. If $m_{a} = 2$, then the path connecting $p_{a+1}$ to another point must intersect the path connecting $p_{a}$ with $p_{n_{a}}$, contradicting the graph is projective planar.

\end{proof}
\begin{figure}[H]
\begin{center}
    \includegraphics[scale=0.52]{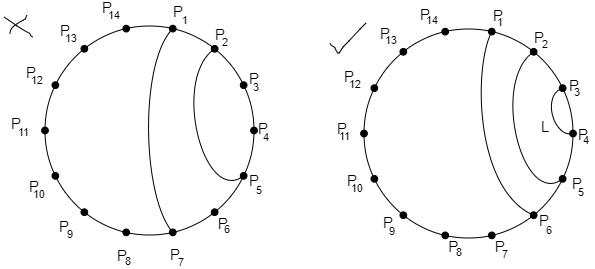}
    \caption{Existence of a path connecting $p_{i}$, $p_{i+1}$ intersecting the boundary only at $p_{i}$, $p_{i+1}$}
\end{center}
\end{figure}

A path connecting $p_{i}$, $p_{i+1}$ that intersects the boundary only at points $p_{i}$, $p_{i+1}$ is called a \textit{half-moon path}.

\begin{theo}\label{Theorem2}
Every 0-homologous cycle in the projective plane is isotopic to the standard affine cycle. Every 1-homologous cycle in the projective plane is isotopic to the standard vertical cycle.
\end{theo}
\begin{proof}
Suppose a cycle $C$ crosses the boundary at $n$ points. Without loss of generality, apply ambient isotopy such that $C$ intersects the boundary only transversely. If $n=0$, $C$ bounds a disk and is isotopic to the standard affine cycle. If $n=1$, $C$ does not bound a disk and is isotopic to the standard vertical cycle.

Otherwise, consider a half-moon path $L$ which exist by Lemma 3.2. Without loss of generality, apply ambient isotopy so that $L$ does not touch other paths. Then take a sufficiently small neighbourhood $U$ that contains only $p_{i}$, $p_{i+1}$, $L$, and the connected pieces of the paths through $p_{i}$, $p_{i+1}$. Isotope the small neighbourhood over the boundary. Note the isotopy preserves whether the cycle bounds a disk, and the cycle after isotopy crosses the boundary at $n-2$ points.

Inductively applying this argument, and the resulting cycle $C'$ crosses the boundary at either 0 points or 1 point. If $C'$ crosses the boundary at 0 points, it bounds a disk and is isotopic to the standard affine cycle. Since $C$ is isotopic to $C'$, every 0-homologous cycle is isotopic to the standard affine cycle. If $C'$ crosses the boundary at 1 point, it does not bound a disk and is isotopic to the standard vertical cycle. Since $C$ is isotopic to $C'$, every 1-homologous cycle is isotopic to the standard vertical cycle.
\begin{figure}[H]
\begin{center}
    \includegraphics[scale=0.4]{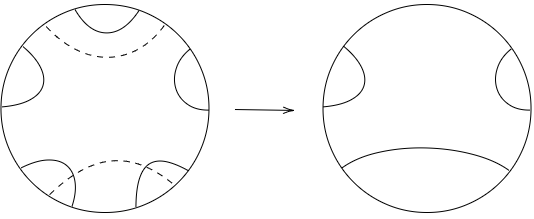}
    \caption{Reduced number of crossing with the boundary}
\end{center}
\end{figure}
\end{proof}
\begin{col}
A 1-homologous cycle crosses the boundary an odd number of times. A 0-homologous cycle crosses the boundary an even number of times.
\end{col}

Next we generalize the following result:

\begin{lem}[Glover et al \cite{Glover}] \label{gloverlemma}
Any two 1-homologous cycles $C_1$ and $C_2$ in the projective plane intersect each other.
\end{lem}

\begin{theo}\label{Theorem3}
Given two different 1-homologous cycles $C_{1}$ and $C_{2}$ in the projective plane that intersect only transversely, the number of crossings must be odd.
\end{theo}

\begin{proof}
Suppose we have two different 1-homologous cycles $C_{1}$ and $C_{2}$ that intersect only transversely. By Theorem \ref{Theorem2}, we may isotope so that one of the cycles, $C_{1}$, is the standard vertical cycle. Without loss of generality, isotope the cycles so that they do not intersect on the boundary. We may assume that $C_{2}$ intersects the boundary at points $p_{1}$, ..., $p_{4k+2}$ where $p_{i}$, $p_{j}$ represent the same point if and only if $ i \equiv j \mod 2k+1$.

\begin{figure}[H]
\begin{center}
    \includegraphics[scale=0.39]{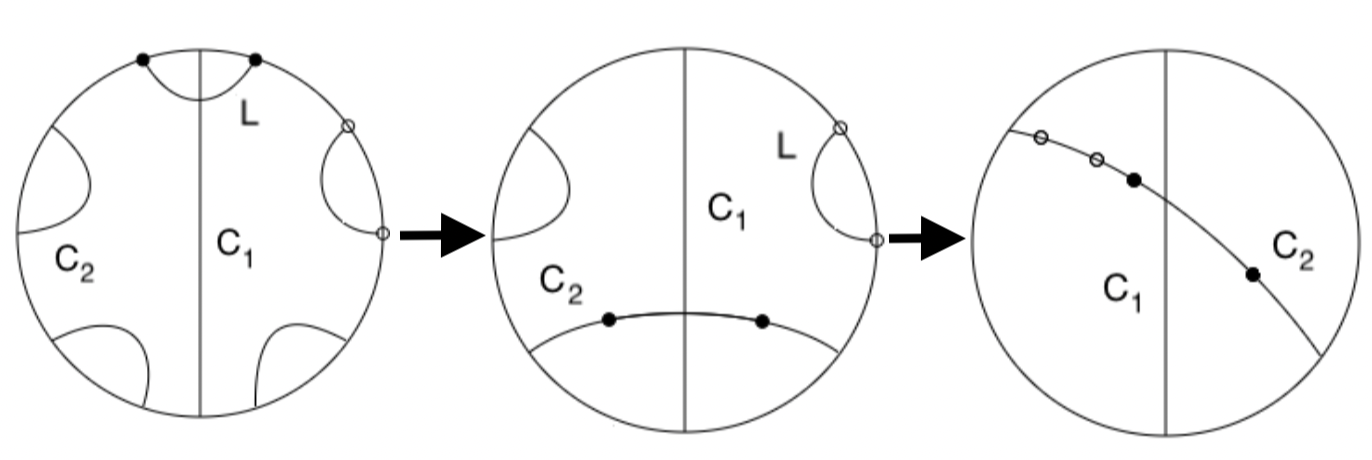}
    \caption{Reduced number of crossings of $C_{2}$ with the boundary by applying isotopy to $L$}
\end{center}
\end{figure}

If $k=0$, since $C_{1}$ and $C_{2}$ intersect only transversely, the number of crossings is odd. Otherwise, consider a half-moon interior path $L$ connecting $p_{i}$, $p_{i+1}$ which exists by Lemma 3.2. Take a sufficiently small neighbourhood $U$ that contains only $p_{i}$, $p_{i+1}$, $L$, and the connected pieces of the paths through $p_{i}$, $p_{i+1}$. Isotope the small neighbourhood over the boundary.

Note the number of crossing of $C_{1}$ with $C_{2}$ is preserved mod 2 while the resulting curve intersects the boundary at $4k-2 = 4(k-1)+2$ points. By induction on $k$, $C_{1}$ and $C_{2}$ can be isotoped to two cycles intersecting the boundary only once. Thus $C_1$ and $C_2$ intersect an odd number of times.  
\end{proof}

\section{Separating projective planar graphs}

The property of being outer-projective-planar is a minor closed property. By Robertson and Seymour's result \cite{robertson}, the set of minor-minimal nonouter-projective-planar graphs is finite. All such 32 graphs were characterized and divided into nine families \cite{Archdeacon}. The family members are related to each other by $\Delta-Y$ exchanges. 

Because outer-projective-planar graphs are nonseparating, nonouter-projective-planar graphs are of interest to us because they are good candidates to produce the beginning of the set of minor-minimal separating projective planar graphs. For each of the 32 graphs, we examined if they were separating. If the graph was nonseparating, we examined the given graph with a vertex or edge added, or a vertex splitting, to find a separating graph. In this way, we have hopefully characterized much of the set of minor-minimal separating projective planar graphs, but we have not yet characterized the complete set.

\subsection{$\alpha$ family}
The $\alpha$ family of minor-minimal nonouter-projective-planar graphs consists of $\alpha_1$, $\alpha_2$, and $\alpha_3$. The graph $\alpha_1$ is two disjoint copies of $K_4$. The graph $\alpha_2$ is two disjoint components: $K_4$ and $K_{3,2}$. The graph $\alpha_3$ is two disjoint copies of $K_{3,2}$.

\begin{proposition}
Every member of the $\alpha$ family is minor-minimal separating projective planar.
\end{proposition}

\noindent The proof of this proposition relies on the following theorem:

\begin{theo}[Halin \cite{Halin}, Chartrand and Harary \cite{CH}]\label{halin}
A graph $G$ is outerplanar if and only if $G$ does not contain $K_4$ nor $K_{3,2}$ as a minor.
\end{theo}

\begin{proof}
Each of the members of the $\alpha$ family have two components. Both of these components contain separating cycles if they are affine embedded with another component. If both components were embedded in the projective plane with 1-homologous cycles, they would intersect, by Lemma \ref{gloverlemma}. This means, in a projective planar drawing, that at least one component has all 0-homologous cycles.

Since the embedding of the component with all 0-homologous cycles is equivalent to an affine embedding of the component by Theorem \ref{Theorem1}, the embedding of the component with all 0-homologous cycles will also contain a separating cycle. Since at least one component of every embedding of the $\alpha$ family has all 0-homologous cycles, every embedding will contain a separating cycle. So, the $\alpha$ family is separating projective planar. 

The graphs $K_4$ and $K_{3,2}$ are minor-minimal nonouterplanar graphs. So, for a minor of a graph in the $\alpha$ family, one component can be affine embedded without a separating cycle. The other component can be embedded with a 1-homologous cycle, with both components together nonseparating. Therefore, the members of the $\alpha$ family are minor-minimally separating projective planar.
\end{proof}

\subsection{$\beta$ family}
The $\beta$ family of minor-minimal nonouter-projective-planar graphs are pictured in Figure \ref{fig:minipage1}-\ref{fig:minipage6}. The $\beta$ family consists of $K_4$ and $K_{3,2}$ copies glued at a vertex.

\begin{proposition}
Every member of the $\beta$ family is nonseparating projective planar.
\end{proposition}

\begin{proof}
Every member in the $\beta$ family has a nonseparating embedding in the projective plane as illustrated in Figure \ref{fig:minipage1} to Figure \ref{fig:minipage6}.
\end{proof}

\begin{figure}[H]
\centering
\begin{minipage}[b]{0.45\linewidth}
    \includegraphics[scale=0.4]{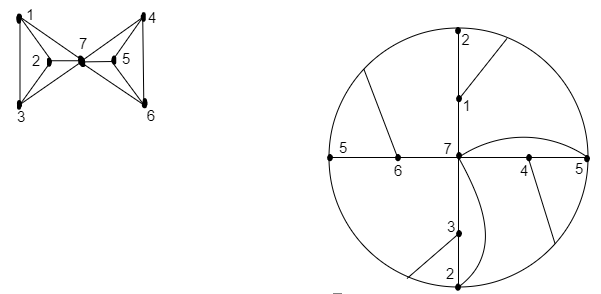}
    \caption{$\beta_1$}
    \label{fig:minipage1}
\end{minipage}
\quad
\begin{minipage}[b]{0.45\linewidth}
    \includegraphics[scale=0.4]{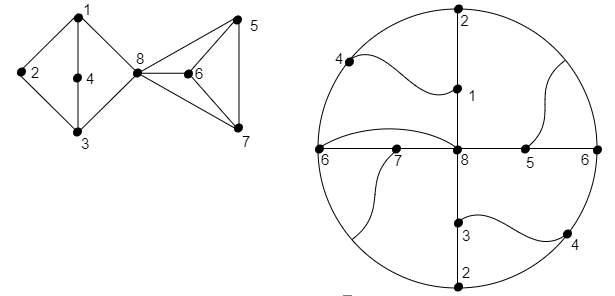}
    \caption{$\beta_2$}
    \label{fig:minipage2}
\end{minipage}
\end{figure}

\begin{figure}[H]
\centering
\begin{minipage}[b]{0.45\linewidth}
    \includegraphics[scale=0.4]{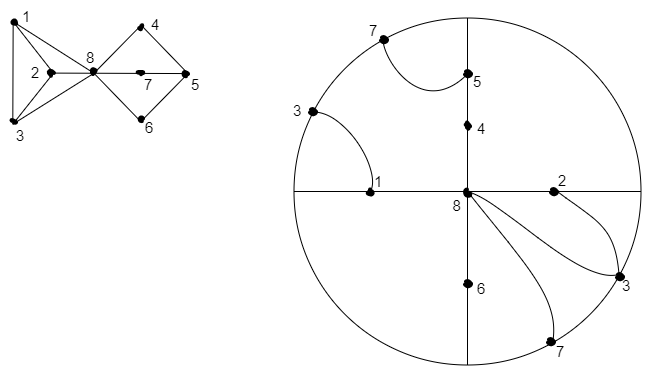}
    \caption{$\beta_3$}
    \label{fig:minipage3}
\end{minipage}
\quad
\begin{minipage}[b]{0.45\linewidth}
    \includegraphics[scale=0.4]{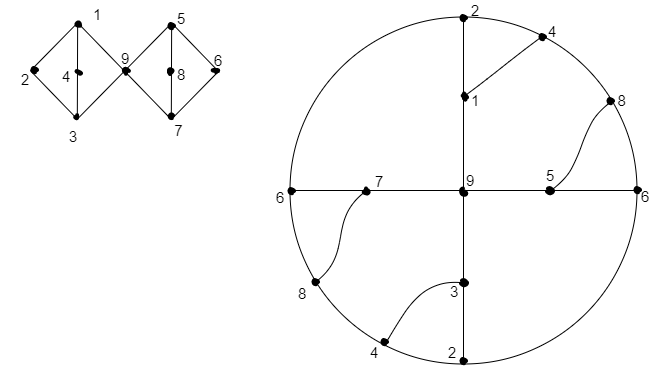}
    \caption{$\beta_4$}
    \label{fig:minipage4}
\end{minipage}
\end{figure}

\begin{figure}[H]
\centering
\begin{minipage}[b]{0.45\linewidth}
    \includegraphics[scale=0.4]{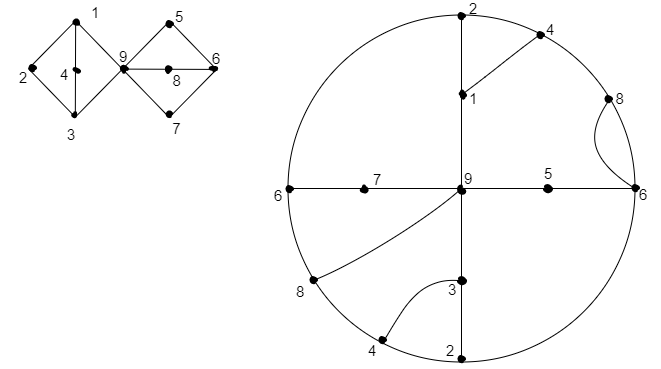}
    \caption{$\beta_5$}
    \label{fig:minipage5}
\end{minipage}
\quad
\begin{minipage}[b]{0.45\linewidth}
    \includegraphics[scale=0.4]{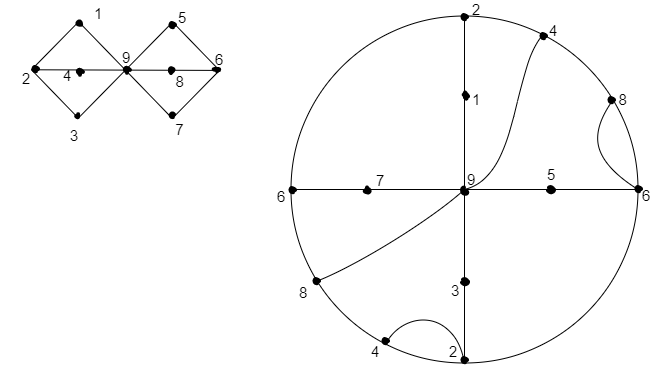}
    \caption{$\beta_6$}
    \label{fig:minipage6}
\end{minipage}
\end{figure}

Note that each embedding in Figure \ref{fig:minipage1} to Figure \ref{fig:minipage6} is weakly separating. In Figure \ref{fig:minipage1}, no topological path connects vertices 2 and 5 without its interior intersecting the graph; in Figure \ref{fig:minipage2}, no path connects vertices 2 and 6; in Figure \ref{fig:minipage3}, no path connects vertices 2 and 7; in Figure \ref{fig:minipage4}, no path connects vertices 2 and 8; in Figure \ref{fig:minipage5}, no path connects vertices 2 and 8; in Figure \ref{fig:minipage6}, no path connects vertices 1 and 8. We conjecture that every member in the $\beta$ family is weakly separating.

\subsection{$\epsilon$ family}
The $\epsilon$ family of minor-minimal nonouter-projective-planar graphs are pictured in Figure \ref{fig:FE1} and \ref{fig:FE2}.
\begin{proposition}
Every member of the $\epsilon$ family is strongly nonseparating, and hence nonseparating projective planar.
\end{proposition}

\begin{proof}
For $\epsilon_1$, $\epsilon_2$, $\epsilon_3$ and $\epsilon_5$ consider the the embeddings in Figure \ref{fig:FE1}. Since all vertices can be connected by paths that intersect the graph only at endpoints, they are strongly nonseparating. For $\epsilon_4$ and $\epsilon_6$, consider the the embeddings in Figure \ref{fig:FE2}. They are also strongly nonseparating projective planar. Therefore all members of the $\epsilon$ family are nonseparating projective planar.
\end{proof}

\begin{figure}[H]
    \centering
    \includegraphics[scale=0.23]{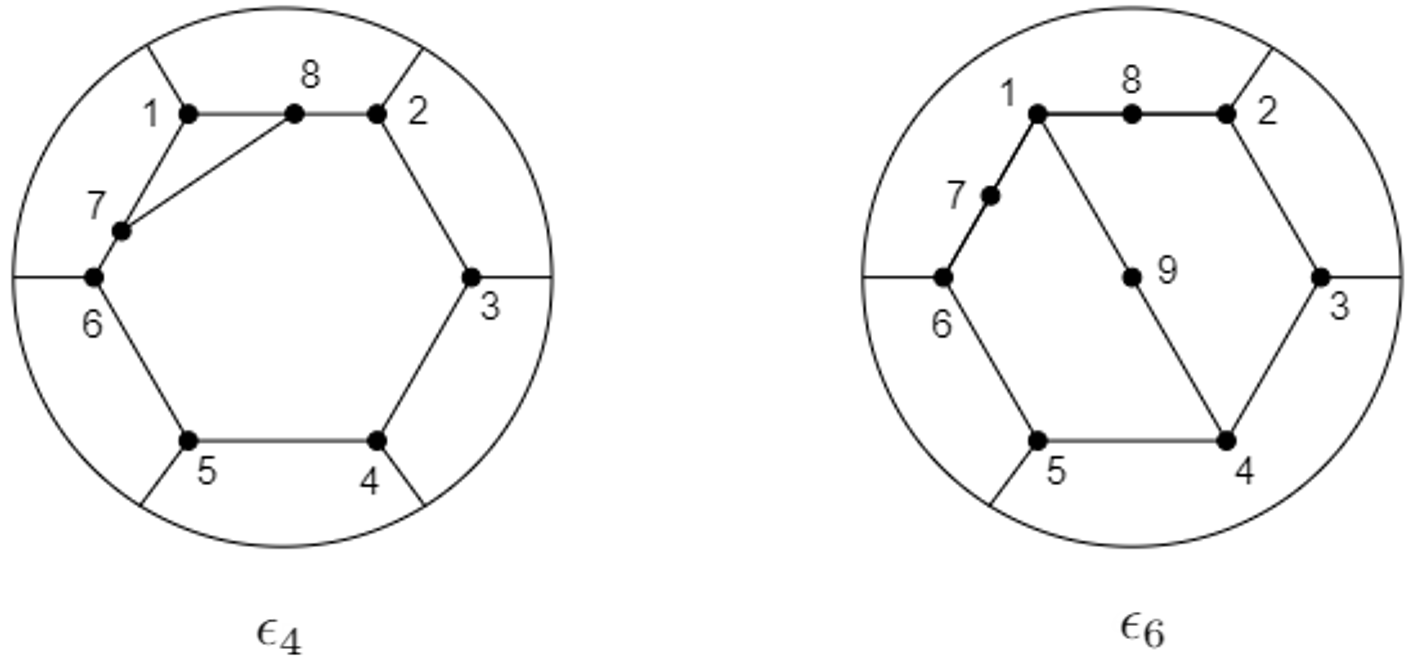}
    \caption{Strongly nonseparating projective planar drawings of $\epsilon_4$ and $\epsilon_6$}
    \label{fig:FE2}
\end{figure}

\begin{figure}[H]
    \centering
    \includegraphics[scale=0.35]{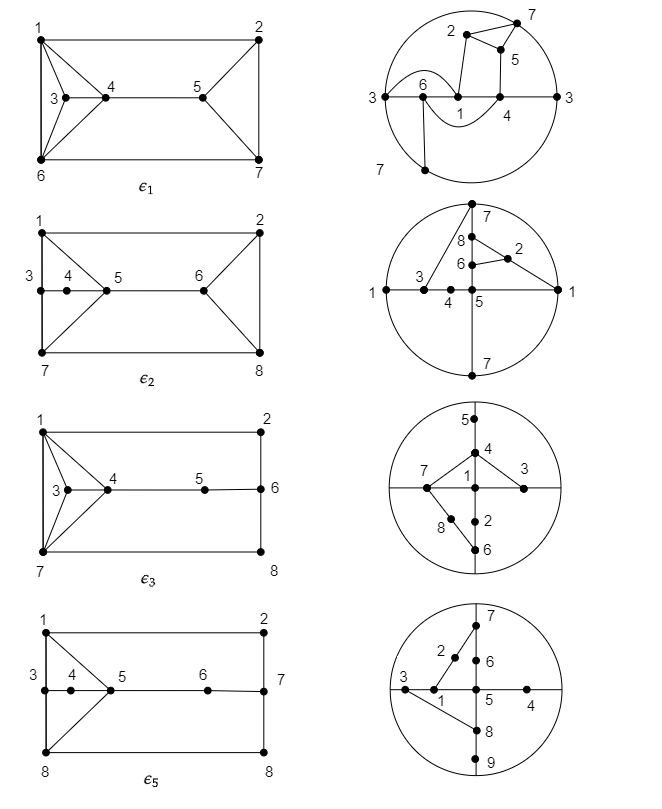}
    \caption{Strongly nonseparating projective planar drawings of $\epsilon_1$, $\epsilon_2$, $\epsilon_3$ and $\epsilon_5$}
    \label{fig:FE1}
\end{figure}

\subsection{$\delta$ family}
The $\delta$ family of minor-minimal nonouter-projective-planar graphs has two graphs shown in Figure \ref{D1D2}.
\begin{center}

\begin{figure}[H]
\begin{center}
\includegraphics[scale=0.18]{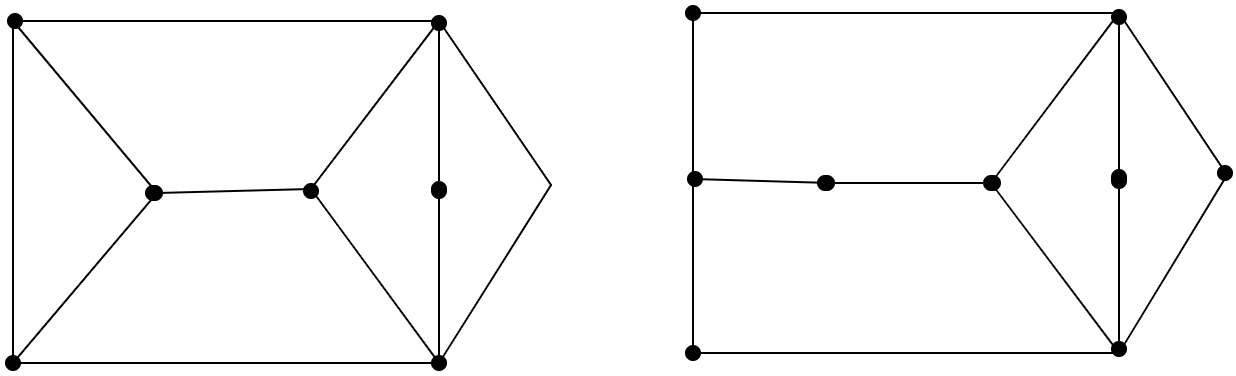}
\caption{The $\delta$ Family includes $\delta_1$ (left) and $\delta_2$ (right)}
\label{D1D2}
\end{center}
\end{figure}

\end{center}

\begin{lem}\label{CycleAddion}
For any two cycles $C_1$ and $C_2$ intersecting along an arc, $D$, define the sum of $C_1$ and $C_2$ to be $(C_1\cup C_2)\setminus D$. The sum of two 0-homologous cycles and the sum of two 1-homologous cycles are 0-homologous cycles, and the sum of a 1-homologous cycle and a 0-homologous cycle is 1-homologous 

\end{lem}

\begin{proposition}
Both graphs in the $\delta$ family are separating.
\end{proposition}
\begin{proof}
By Theorem 3.1, a 2-connected planar graph embedded in the projective plane with all cycles 0-homologous can be isotoped to an affine embedding. Namely, if a 2-connected subgraph $H$ of $\delta_1$ or $\delta_2$ contains either $K_4\dot\cup K_1$ or $K_{3,2}\dot\cup K_1$ as a minor, then any embedding of $H$ with all cycles 0-homologous is separating because such an embedding can be isotoped to an affine embedding containing either $K_4\dot\cup K_1$ or $K_{3,2}\dot\cup K_1$ \cite{Dehkordi}. 
Since every cycle bounding a face $F$ in $\delta_1$ or $\delta_2$ can be viewed as the sum of the cycles that bound every face except for $F$, every embedding of $\delta_1$ and $\delta_2$ has an even number of 1-homologous cycles. Thus, since all 1-homologous cycles intersect each other and both $\delta_1$ and $\delta_2$ have reflection symmetry, the number of cases we must consider is considerably small. 
All projective planar embeddings of $\delta_1$ and $\delta_2$, up to symmetry and the homology class of each cycle, are represented in Figure \ref{AllEmbeddings}, where a face is shaded if and only if its bounding cycle is 1-homologous cycle in the corresponding embedding. The highlighted subgraphs are 2-connected with all 0-homologous cycles that contain either $K_4\dot\cup K_1$ or $K_{3,2}\dot\cup K_1$ as minors, and since one exists for every embedding of $\delta_1$ and $\delta_2$, both graphs are separating.
\end{proof}
\begin{figure}[H]
\centering
\includegraphics[scale=0.4]{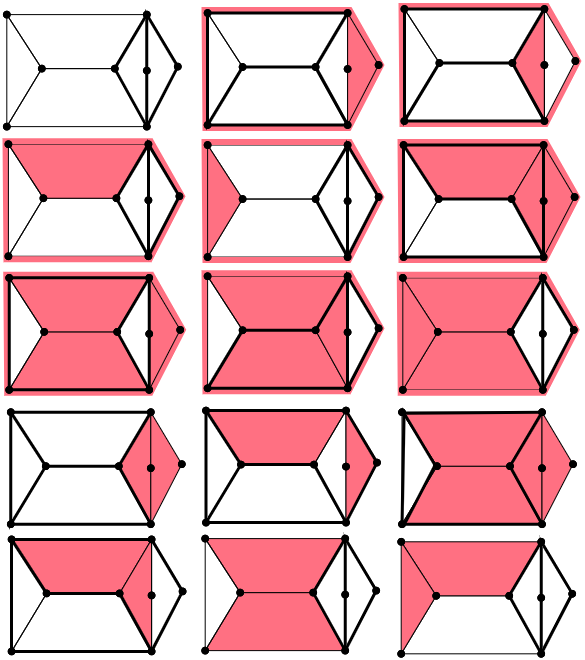}
\includegraphics[scale=0.4]{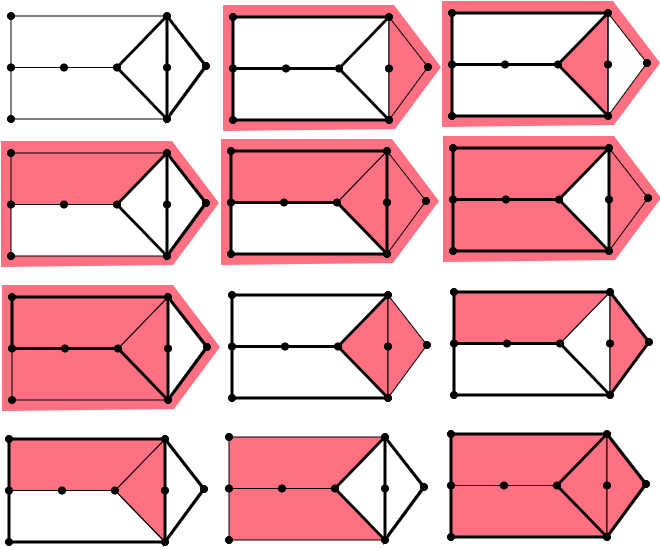}
\caption{Projective planar embeddings of $\delta_1$ and $\delta_2$}
\label{AllEmbeddings}
\end{figure}

\subsection{$\gamma$ family}
The $\gamma$ family of minor-minimal nonouter-projective-planar graphs are pictured in Figure \ref{gammafamily}.

\begin{proposition}
Every graph in the $\gamma$ family, excluding $\gamma_{6}$, is strongly nonseparating and thus nonseparating
\end{proposition}
\begin{proof}
Every graph in the $\gamma$ family, excluding $\gamma_{6}$, has a strongly nonseparating embedding in the projective plane as illustrated in Figure \ref{gammafamily}.
\end{proof}

\begin{figure}[H]
\begin{center}
    \includegraphics[scale=0.62]{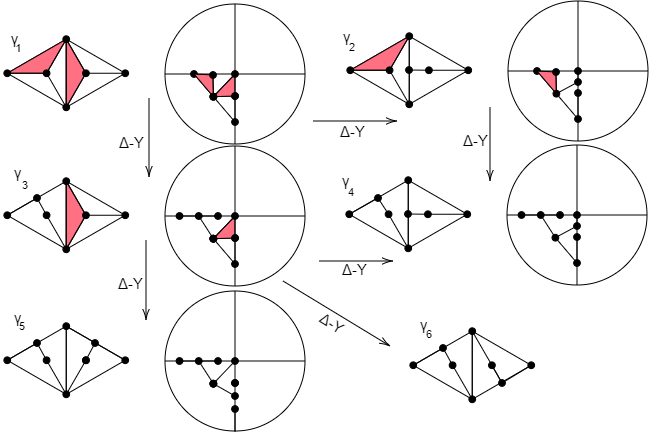}
\caption{Every graph in the $\gamma$ family, excluding $\gamma_{6}$, is strongly nonseparating and thus nonseparating}
\label{gammafamily}
\end{center}
\end{figure}

\subsubsection{$\gamma_{6}$}
The graph $\gamma_{6}$ is obtained from $\gamma_{1}$ via two $\Delta - Y$ exchanges. 

\begin{figure}[H]
\begin{center}
    \includegraphics[scale=0.45]{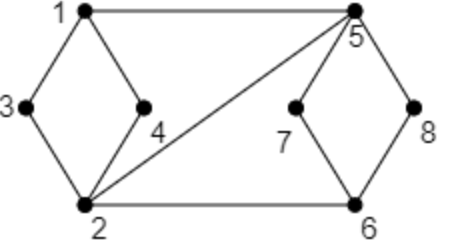}
\caption{The graph $\gamma_{6}$ with vertices labelled}\label{Gamma}
\end{center}
\end{figure}
In a given embedding, a \textit{0-homologous path} is a path that crosses the boundary an even number of times. Similarly, in a given embedding, a \textit{1-homologous path} is a path that crosses the boundary an odd number of times.

\begin{proposition}
The graph $\gamma_6$ is separating projective planar.
\end{proposition}

\begin{proof}
Label the vertices of $\gamma_6$ as in Figure~\ref{Gamma}. Note that the planar embedding of $\gamma_6$ is separating and that there are 5 paths connecting vertex 1 with vertex 2, including $P_{1} = (1, 3, 2)$, $P_{2} = (1,4,2)$, $P_{3} = (1,5,2)$, $P_{4} = (1,5,7,6,2)$, and $P_{5} = (1,5,8,6,2)$. Embed $\gamma_6$ in the projective plane. By the pigeonhole principle, either at least three of these paths are 0-homologous or at least three paths are 1-homologous. If at least three paths are 1-homologous, isotope a sufficiently small neighbourhood containing vertex $1$ and no other vertex over the boundary. Then every 0-homologous path becomes a 1-homologous path and every 1-homologous path become a 0-homologous path. In the resulting embedding, there are at least three 0-homologous paths.

Suppose at least three paths are 0-homologous, then the other two paths may be both 0-homologous, one 0-homologous and one 1-homologous, or both 1-homologous.

\begin{itemize}
    \item If both of the two other paths are 0-homologous, then the paths $(1,5)$, $(5,2)$, $(5,7,6,2)$, $(5,8,6,2)$ must either be all 0-homologous or all 1-homologous. Therefore every cycle in the embedding is 0-homologous. The embedding can be isotoped into an affine embedding and is therefore separating.
    
    \item If both of the two other paths are 1-homologous, separately consider cases according to where these two paths lie. Note that not every pair of paths can be the two 1-homologous paths: suppose $P_{1}$, $P_{5}$ are the two 1-homologous paths, then the cycles $(1,4,3,2)$ and $(5,8,6,7)$ are two 1-homologous cycles with no intersection, contradicting Lemma \ref{gloverlemma}. 
    
    By symmetry there are only two cases, with the two 1-homologous paths being $P_{1}$, $P_{2}$ or $P_{1}$, $P_{3}$. In each of the two cases, the three 0-homologous paths together contains a minor equivalent to $K_{3,2}$ which is nonouterplanar. Therefore a vertex must be on one side of the 0-homologous cycle formed by the other two paths, and the two 1-homologous paths must lie on the other side. In particular, at least one of the two paths has a vertex on it. Therefore the cycle is separating.
    
    \item If exactly one of the two other paths is 1-homologous, separately consider cases according to where the 1-homologous path lies. If the 1-homologous path is $P_{1}$ or $P_{2}$, consider the minor obtained by deleting $(1,3,2)$ or $(1,4,2)$. If the 1-homologous path is $P_{3}$, consider the minor obtained by deleting $(5,2)$. If the 1-homologous path is $P_{4}$ or $P_{5}$, consider the minor obtained by deleting $(5,7,6)$ or $(5,8,6)$. In each case, the resulting minor has only 0-homologous cycles and can be isotped to an affine embeddings. In addition, in each case the resulting minor has a subgraph of $K_1 \dot\cup K_{3,2}$. Therefore the graph has a separating cycle.

    \begin{figure}[ht]
\centering
\begin{minipage}[b]{0.4\linewidth}
			  \centering
    \includegraphics[scale=0.6]{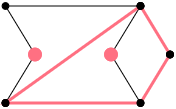}
    \caption{$C_{1}$ is the 1-homologous path}
    \label{fig:C1is1homologous}
\end{minipage}
\quad
\begin{minipage}[b]{0.4\linewidth}
			  \centering
    \includegraphics[scale=0.6]{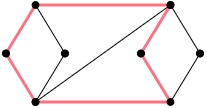}
    \caption{$C_{3}$ is the 1-homologous path}
    \label{fig:C3is1homologous}
\end{minipage}
\end{figure}

\end{itemize}

Therefore the graph $\gamma_{6}$ is separating.
\end{proof}

\subsection{$\zeta$ family}
The members of the $\zeta$ family of minor-minimal nonouter-projective-planar graphs are shown in Figures \ref{ZetaNot3} and \ref{Zeta3}.

\begin{theo}
Besides $\zeta_3$, every member of the $\zeta$ family is strongly nonseparating.
\end{theo}
\begin{proof}
To show that $\zeta_1, \zeta_2, \zeta_4, \zeta_5,$ and $\zeta_6$ are strongly nonseparating, it suffices to show a single strongly nonseparating embedding of each graph. See Figure \ref{ZetaNot3}.

\begin{figure}[H]
    \centering
    \subfigure[]{\includegraphics[width=0.3\textwidth]{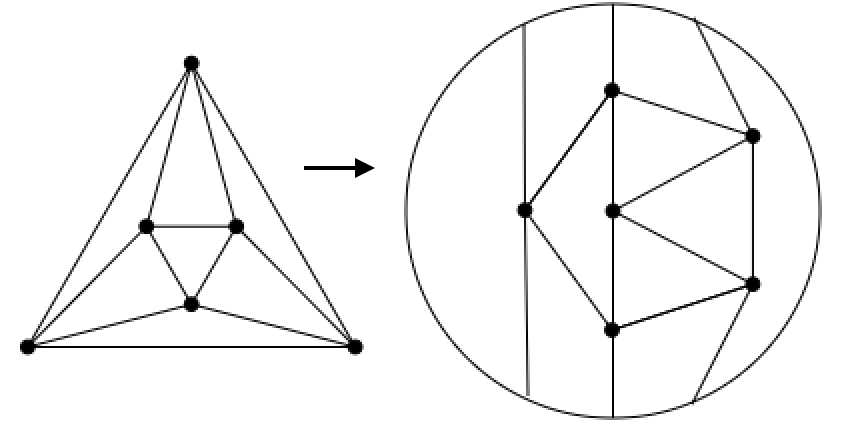}} 
    \subfigure[]{\includegraphics[width=0.3\textwidth]{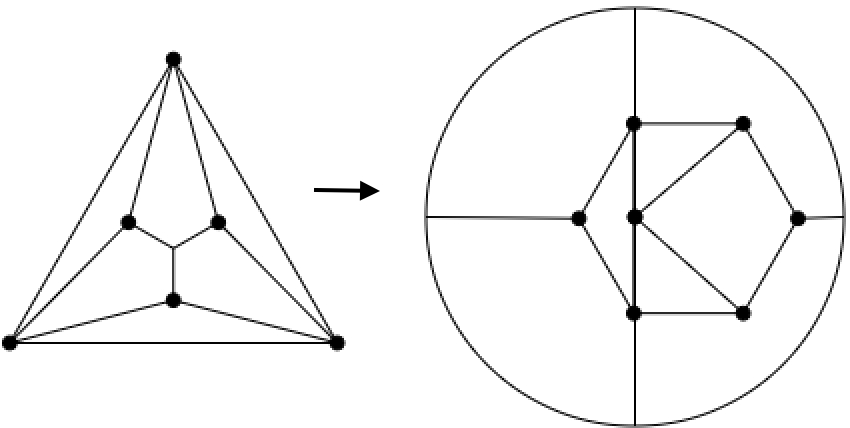}} 
    \subfigure[]{\includegraphics[width=0.3\textwidth]{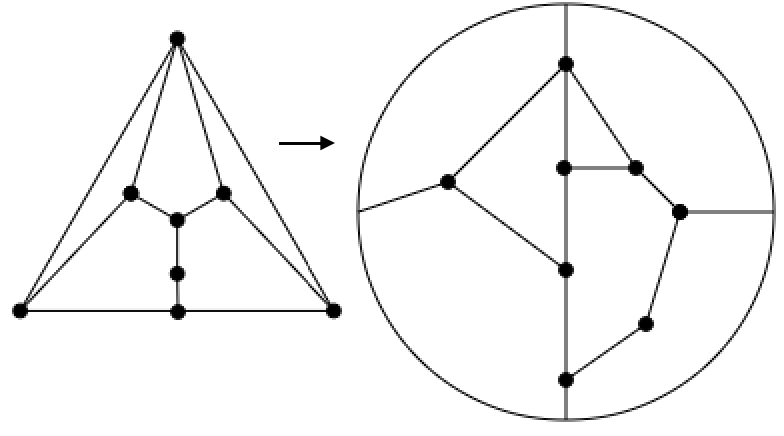}}
    \subfigure[]{\includegraphics[width=0.3\textwidth]{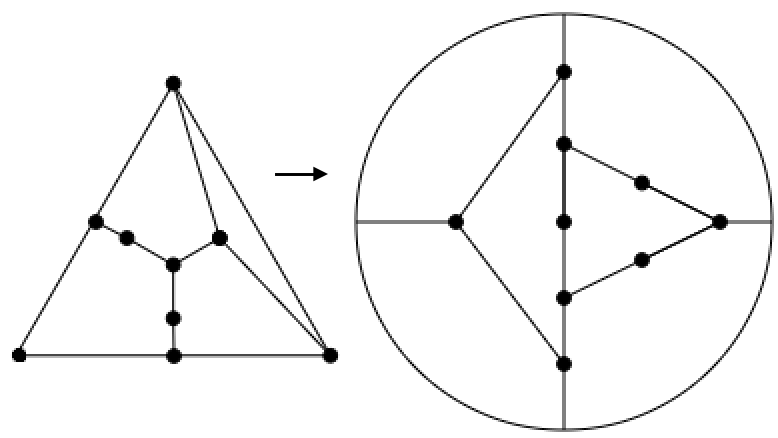}}
    \subfigure[]{\includegraphics[width=0.3\textwidth]{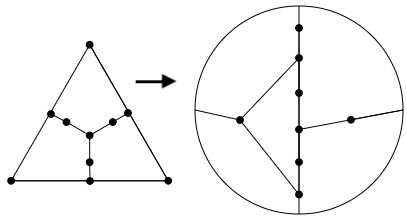}}
    \caption{strongly nonseparating embeddings of $\zeta_1, \zeta_2, \zeta_4, \zeta_5$, and $\zeta_6$}
    \label{ZetaNot3}
\end{figure}


We now show that $\zeta_3$, the vertices and edges of a cube, is weakly separating, and indeed separating.

Label the faces in the planar embedding of $\zeta_3$ as in Figure \ref{Zeta3}.
\begin{center}
\begin{figure}[H]
\begin{center}
\includegraphics[scale=0.17]{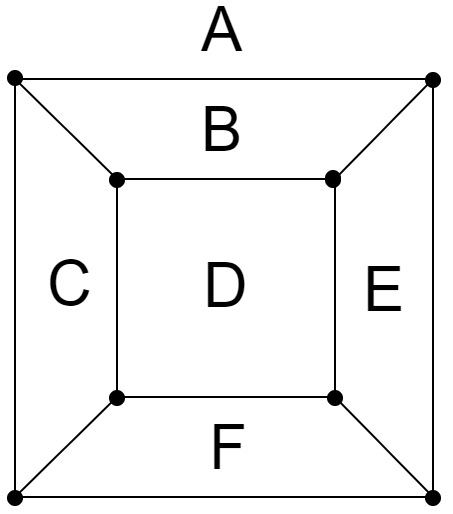}

\caption{A planar embedding of $\zeta_3$ with faces labelled}
\label{Zeta3}
\end{center}
\end{figure}
\end{center}

It is well known that for any cycle $X$ in any drawing of $\zeta_3$ in ${\mathbb{R}P}^2$, the homology class of $X$ is equal to the modulo 2 sum of the homology classes of the boundaries of the faces on one side of $X$ in the above drawing. Let $D^*$ be a drawing of $\zeta_3$. 

The cube has $24$ symmetries, and in particular we may first choose any of the $6$ faces to be in the place of $A$ and then choose any of the $4$ faces adjacent to the first to be in the place of $B$. There are two case to consider: 

\begin{enumerate}
    \item Suppose every cycle in $D^*$ is homologous to zero. Then, by Theorem \ref{Theorem1}, $D^*$ may be isotoped to an affine embedding. Since $\zeta_3$ contains $K_4\dot\cup K_1$ as a minor, any affine embedding of $\zeta_3$ is separating. Thus, $D^*$ is separating.
    \item  Let $f$ map the planar embedding of $\zeta_3$ to $\mathbb{R}P^2$. Without loss of generality, suppose $f(Bd(A))$ is homologous to one. Then $f(Bd(D))$ may not be homologous to one since all one homologous cycles in $D^*$ intersect, so as $Bd(A)$ bounds the rest of the graph in the above embedding, an odd number of the boundaries of $B$, $C$, $E$, or $F$ is also homologous to one in $\mathbb{R}P^2$. Without loss of generality, suppose $f(Bd(B))$ is homologous to one. Then $f(Bd(F))$ is not homologous to one because $f(Bd(F))$ does not intersect $f(Bd(B))$. Thus, since $f(Bd(C))$ and $f(Bd(E))$ can not both be homologous to one, neither may be. This means that $f(Bd(A))$ and $f(Bd(B))$ are the only one homologous face bounding cycles. The drawing of $f(Bd(A))$ and $f(Bd(B))$ as one homologous cycles separates $\mathbb{R}P^2$ into two connected regions and contains six of the eight vertices of $\zeta_1$. Since the remaining two vertices are connected by an edge, they must be in the same region. This leaves one option for $D^*$ as in Figure \ref{Zeta3Sep}, where the boundary of $C,D,F$ is a separating cycle.
\end{enumerate}

\begin{figure}[H]
\centering
\includegraphics[scale=.35]{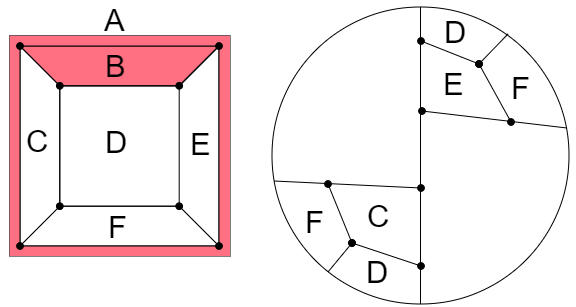}
\caption{Up to symmetry, the only projective planar embedding of $\zeta_1$ with a one-homologous cycle}
\label{Zeta3Sep}
\end{figure}

\end{proof}

\subsection{$\eta$ family}

The $\eta$ family of minor-minimal nonouter-projective-planar graphs contains only one member, pictured in Figure \ref{etaaffine}.

\begin{figure}[H]`
\centering
\includegraphics[scale=0.27]{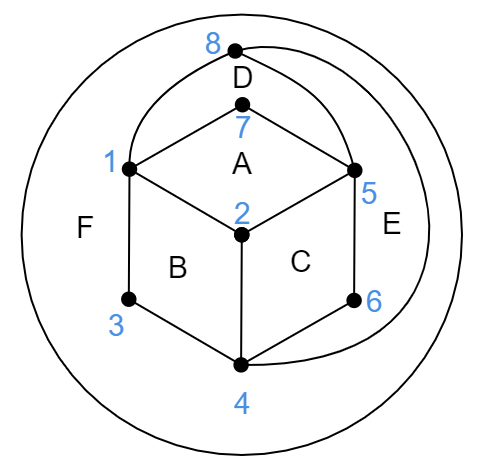}
\caption{An affine embedding of $\eta_1$}
\label{etaaffine}
\end{figure}

A \textit{subdivision} of a graph $G$ is a graph resulting from the subdivision of the edges of $G$. If an edge $e$ with endpoints $v_1$ and $v_2$ has a subdivision, this creates a graph identical to $G$ except with a new vertex $v_3$, where the edge $e$ is replaced by the two edges $(v_1,v_3)$ and $(v_3, v_2)$. A \textit{0-homologous region} is a face of an embedded graph that is bounded by a 0-homologous cycle.

\begin{proposition}
The graph $\eta_1$ is separating projective planar.
\end{proposition}

\begin{proof}
Consider the affine embedding of $\eta_1$ shown in Figure 
\ref{etaaffine}, with regions labelled $A$ through $F$. It is unique, as $\eta_1$ is a subdivision of a 3-connected graph \cite{Whitney}.

The affine embedding has a separating cycle, as evident from the drawing in Figure \ref{etaaffine}. Observe that in an embedding of $\eta_1$, if the subgraphs of $\eta_1$ in Figure 23 are affine embedded, there is already a separating cycle in the embedding, since there is an affine embedded $K_{3,2} \dot\cup K_{1}$, or subdivision.

\begin{figure}[H]
\centering
\includegraphics[scale=0.28]{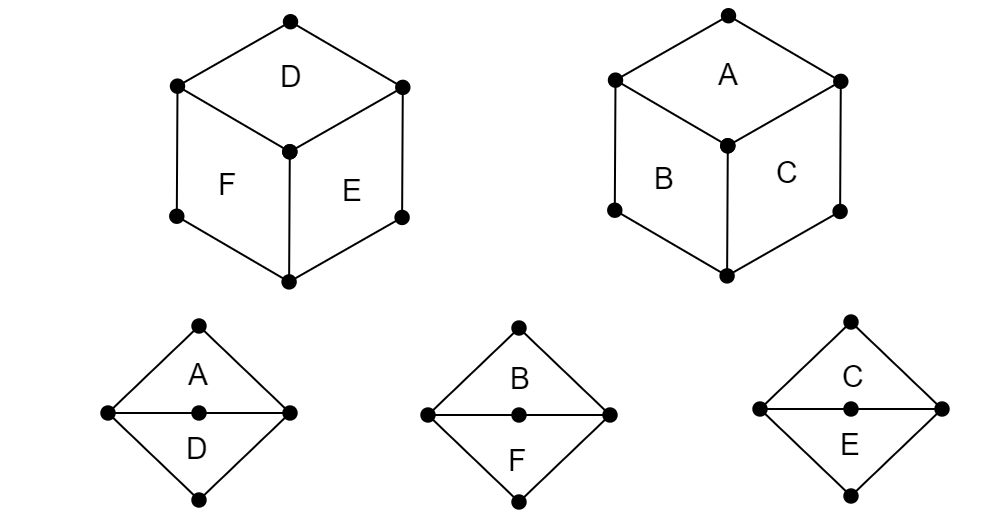}
\caption{Separating subgraphs of $\eta_1$}
\label{sepsubgraph}
\end{figure}

We know from Theorem \ref{Theorem1} that embeddings in which every cycle is 0-homologous are equivalent to affine embeddings. This means we should consider the possible combinations of regions that can be embedded with a 0-homologous cycle boundary, without creating these subgraphs.

We will show there is no projective planar embedding with no regions bounded by 0-homologous curves. We will begin by supposing the region boundaries of $A$, $B$, $C$, and $D$ are embedded as 1-homologous cycles, as seen in the far left graph of Figure \ref{fig:eta}. Thus far, there are no regions bounded by 0-homologous curves. However, to complete the graph, we must connect vertex 8 to vertex 4. In doing this, $E$ and $F$ are embedded with boundaries that are 0-homologous cycles. Thus, there is no way to embed all region boundaries as 1-homologous cycles. 

\begin{figure}[H]
    \centering
    \includegraphics[scale=0.28]{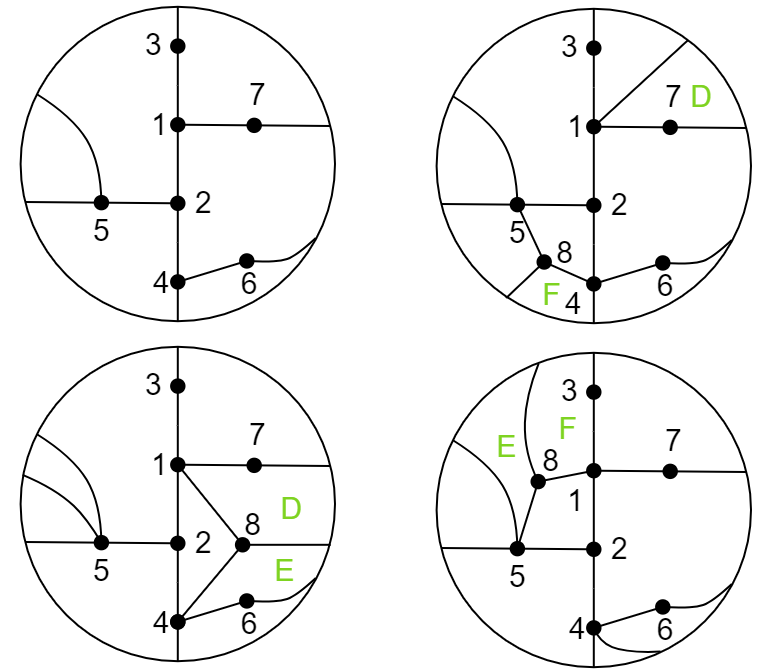}
    \caption{As seen in the top left graph, if we embed region boundaries of $A$, $B$, $C$, and $D$ as 1-homologous cycles, three regions remain for the vertex 8 and its adjacent edges to be embedded into. The three other graphs show the three cases of when vertex 8 and its adjacent edges are embedded in these three regions. This will produce the regions $E$ and $F$, $D$ and $F$, or $D$ and $E$ with region boundaries that are 0-homologous cycles.}
    \label{fig:eta}
\end{figure}

Every region is equivalent to all the others. Without loss of generality, choose region $A$ to be embedded with a region boundary that is 0-homologous. This means we cannot also have the region boundary of $D$ be embedded as a 0-homologous cycle, because that creates one of the separating subgraphs. This also means we cannot have both the region boundaries of $B$ and $F$ also be 0-homologous, nor $C$ and $E$, because this would create the other separating subgraphs. We also cannot have $A$, $B$, and $C$ all have 0-homologous region boundaries - since we have supposed the region boundary of $A$ is 0-homologous, we cannot have both the region boundaries of $A$ and $B$ also be 0-homologous.

There are no embeddings that have just the region boundary of $A$ as a 0-homologous cycle. If so, we have in homology $[A+B+C]=[D+E+F]$ results in $0=1$. For the same reason, there are no embeddings with precisely three or five region boundaries as 0-homologous cycles. If the region boundaries of $A$ and $B$ are 0-homologous, with a 1-homologous region boundary for $C$, then two or zero of the boundaries of regions $D$, $E$, and $F$ are 0-homologous. There cannot be four regions with boundaries that are 0-homologous, because that would create a separating subgraph, so the only remaining case is two chosen regions to have 0-homologous boundaries.

The possible combinations of regions to have 0-homologous boundaries are $AB$, $AC$, $AE$, and $AF$. The combinations $AB$ and $AC$ are equivalent, as are $AE$ and $AF$. Without loss of generality, we embed $AB$ with both 0-homologous region boundaries. We will show there is only one such embedding, pictured in Figure \ref{AB affine}. The embedding is separating.

Now, we will show this embedding is unique up to equivalence. First, affine embed regions $A$ and $B$. Now embed vertex 6. We must connect this vertex to vertices 4 and 5. If neither of these edges intersects the boundary of the projective plane, this will either create region $C$ as a 0-homologous cycle or a region bounded by the cycle $\{5,6,4,3,1,7\}$ as a 0-homologous cycle. We cannot have region $C$ as a 0-homologous cycle, and if we have the cycle $\{5,6,4,3,1,7\}$ as a 0-homologous cycle, we cannot connect vertex 8 to its neighbours without an edge crossing. Thus, the cycle $\{2,4,6,5\}$ must be embedded as a 1-homologous cycle. We must connect vertex 8 to vertices 1, 4, and 5, and there is only one way to do that. Thus, we can conclude the embedding is unique.

\begin{figure}[H]
\centering
\includegraphics[scale=0.28]{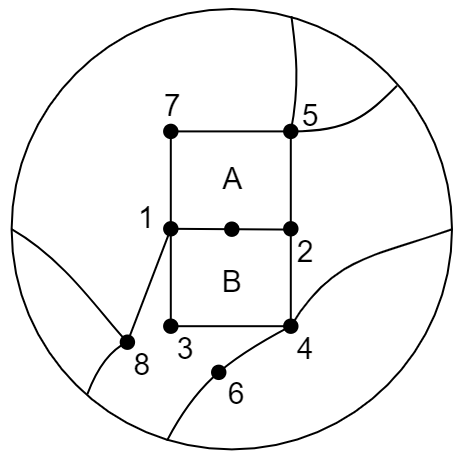}
\caption{Only regions $A $ and $B$ are 0-homologous. The vertices 3 and 7 are separated.}
\label{AB affine}
\end{figure}

The only embedding of the graph that contains $A$ and $E$ as 0-homologous cycles is shown in Figure \ref{AE affine}. The embedding is separating.

As in Figure 26, consider the embedding where the region boundary of $E$ is an affine embedded cycle, and the region boundary of $A$ is a 0-homologous cycle created using two 1-homologous paths. To do this, embed the cycle $\{5, 8, 1, 2\}$ as a standard vertical cycle, and then embed the cycle $7,1,8,5$ as a 1-homologous cycle as well. Now, connect vertices 2, 4 and vertices 4, 8 with an edge - there is only one way to do this. This creates a face bounded by the cycle $\{8,4,2,1\}$. Now we must embed vertex 3. Assume we embed it inside the face bounded by the cycle $\{8,4,2,1\}$. Now we must connect vertex 3 to vertices 1 and 4. This creates the 0-homologous regions $B$ and $F$. Thus, we must embed vertex 3 outside of that face, as well as outside of region $A$ and $E$. Now connect vertex 3 to vertices 1 and 4 in that face. This is the only possible embedding.

\begin{figure}[H]
\centering
\includegraphics[scale=0.26]{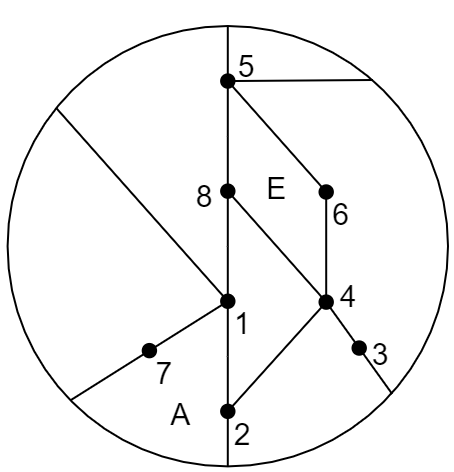}
\caption{Only regions $A$ and $E$ are 0-homologous. Region $A$ is formed by two 1-homologous paths. The vertices 2 and 6 are separated.}
\label{AE affine}
\end{figure}

We can conclude that every embedding of $\eta_1$ is separating.
\end{proof}

\subsection{$\theta$ family}

The $\theta$ family of minor-minimal nonouter-projective-planar graphs has only one member, $\theta_{1}$, which is $K_{5,2}$. Take the vertex set of $\theta_{1}$, or $K_{5,2}$, to be $V = \{ v_{1}, v_{2} \}$ union $U = \{ u_{1}, u_{2}, u_{3}, u_{4}, u_{5} \}$. 

\begin{figure}[H]
\begin{center}
    \includegraphics[scale=0.45]{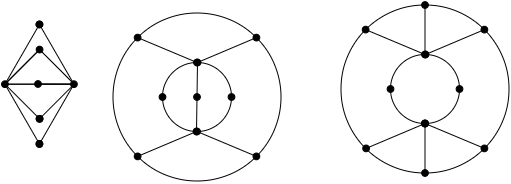}
    \caption{$\theta_{1}$ and two embeddings in the projective plane}
\end{center}
\end{figure}



\begin{proposition}
The graph $\theta_{1}$ is separating projective planar.
\end{proposition}

\begin{proof}
Embed $\theta_{1}$ in the projective plane. Consider the five paths connecting $v_{1}$, $v_{2}$ and note that every cycle consists of exactly two of the five paths. If all five paths are all 0-homologous or 1-homologous, then every cycle is 0-homologous and the embedding can be isotoped into an affine embedding. Since $K_{5,2}$ contains $K_{3,2} \dot\cup K_{1}$ as a subgraph, the affine embedding has a separating cycle.

Otherwise by pigeonhole, either at least three paths connecting $v_{1}$, $v_{2}$ are 0-homologous or at least three paths connecting $v_{1}$, $v_{2}$ are 1-homologous. If at least three paths are 1-homologous, isotope a sufficiently small disk containing $v_{1}$ and no other vertex over the boundary. Then every 0-homologous path becomes a 1-homologous path and every 1-homologous path become a 0-homologous path. In the resulting embedding, there are at least three 0-homologous paths.

The graph formed by the three 0-homologous paths is equivalent to $K_{3,2}$, where one vertex $w_{0}$ lies in the disk bounded by the cycle $C$ formed by two other paths. Consider that there is at least one 1-homologous path connecting $v_{1}$, $v_{2}$ and that a vertex $w_{1}$ lies on this 1-homologous path. Since this 1-homologous path only intersects the three 0-homologous paths at $v_{1}$ and $v_{2}$, $w_{1}$ does not lie in the affine disk bounded $C$. Therefore $C$ separates $w_{0}$, $w_{1}$.
\end{proof}

\begin{proposition}
The graph $\theta_{1}$ is minor-minimal separating projective planar.
\end{proposition}
\begin{proof}
The diagram below illustrates embeddings of key minors of $\theta_1$, which are all nonseparating. These embeddings are $\theta_{1}-e$, $\theta_{1}$ with an edge contraction, $\theta_{1}-u_{i}$ and $\theta_{1}-v_{i}$ from left to right. 

\begin{figure}[H]
\begin{center}
    \includegraphics[scale=0.5]{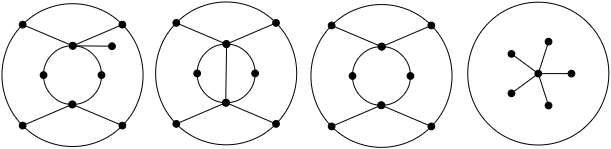}
    \caption{Key minors of $\theta_{1}$}
\end{center}
\end{figure}

\end{proof}

\subsection{$\kappa$ family}
The $\kappa$ family of minor-minimal nonouter-projective-planar graphs has only one member, $\kappa_{1}$.

\begin{proposition}
The graph $\kappa_{1}$ is strongly nonseparating. 
\end{proposition}
\begin{proof}
Figure \ref{Kappa} illustrates a strongly nonseparating embedding of $\kappa_{1}$ in the projective plane, where antipodal points are identified. 
\end{proof}

\begin{figure}[H]
\begin{center}
    \includegraphics[scale=0.25]{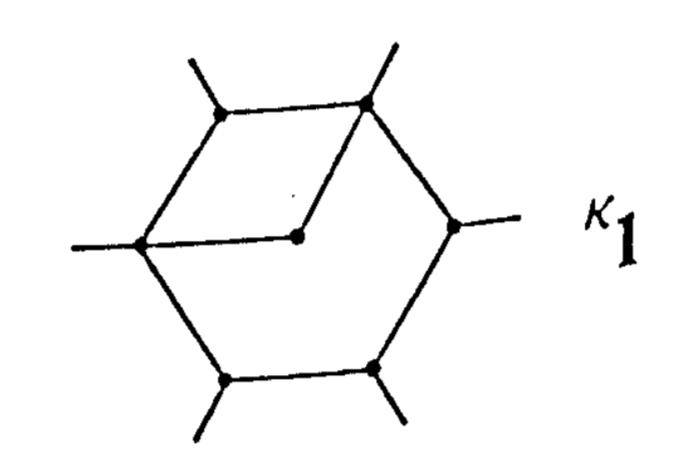}
    \caption{A nonseparating embedding of $\kappa_1$ \cite{Archdeacon}.}
    \label{Kappa}
\end{center}
\end{figure}

\subsection{Weakly separating graphs}

\begin{theo}\label{notvertexconnected}
Let $G$ be a nonouter-projective-planar graph. Then $G\dot\cup K_1$ is weakly separating.
\end{theo}

\begin{proof} Suppose $K_1 = \{w\}$.
Let $G$ be a nonouter-projective-planar graph and let $D$ be a projective planar drawing of $G\dot\cup w$. Then $w\in F$ for some face $F$ of $D$. Since $G$ is nonouter-projective-planar, there is a vertex $v$ of $D$ such that $v\not\in Bd(F)$ and $v\not\in F$. Thus, the component of $v$ in $[\mathbb{R}P^2\setminus D]\cup\{w, v\}$ is a subset of $\mathbb{R}P^2\setminus F$. Since $F$ is a face, $F$ is the component of $w$. Since the path component of any point is a subset of its component, the path components of $w$ and $v$ are disjoint in $\mathbb{R}P^2\setminus D$. Thus, every path from $w$ to $v$ intersects $D$, not just at its endpoints. Hence, $D$ is weakly separating. Since $D$ was arbitrary, $G\dot\cup w$ is weakly separating.
\end{proof}

Here we have an example of a weakly separating projective planar graph that is a nonseparating projective planar graph. For example, $\beta_4$ is a nonouter-projective-planar graph, so, if we add a vertex to $\beta_4$, this new graph is a weakly separating projective planar graph, and the following embedding does not have separating cycles:

\begin{figure}[H]
\begin{center}
    \includegraphics[scale=0.4]{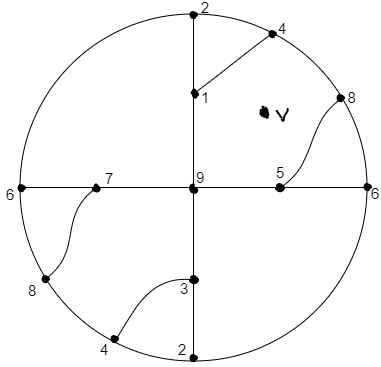}
    \caption{$\beta_4 \dot{\cup} K_1$}
\end{center}
\end{figure}

\begin{theo}\label{stronglynonsep&nonouterpp}
If $G$ is a strongly nonseparating and minor-minimally nonouter-projective-planar graph, then $G\dot\cup K_1$ is minor-minimally weakly separating.
\end{theo}

\begin{proof} Let $G$ be a minor-minimally nonouter-projective-planar strongly nonseparating graph. By Theorem \ref{notvertexconnected}, $G\dot\cup K_1$ is weakly separating. Let $H$ be a proper minor of $G\dot\cup K_1$. Then $H=G'\dot\cup K'$, where $G'$ is a minor of $G$ and $K'$ is a minor of $K_1$. Since $H$ is a proper minor, either $G'$ or $K'$ is a proper minor.

\begin{enumerate}
    \item If $G'$ is a proper minor of $G$, then since $G$ is minor-minimally nonouter-projective-planar, there is a drawing of $G'$ with every vertex on the boundary of the same face. Thus, there is a drawing $D$ of $H$ with $K'$ on the interior of one such face of $G'$. Hence, there is a face $F$ of $H$  with every vertex of $H$ on its boundary. Hence, $H$ is strongly nonseparating. 
    \item If $K'$ is a proper minor of $K_1$, then $K'$ is the empty graph. Thus, $H=G'$. Since $G'$ is a minor of $G$ and strongly nonseparating is a minor-closed property, $H$ is strongly nonseparating.
\end{enumerate}

In both cases, $H$ is strongly nonseparating. Since $H$ was arbitrary, every proper minor of $G\dot\cup K_1$ is strongly nonseparating. Therefore, $G\dot\cup K_1$ is minor-minimally weakly separating. 

\end{proof}

An example of Theorem \ref{stronglynonsep&nonouterpp} would be the if $G$ was the graph $\kappa$ from Section 4.9. We know $\kappa \dot\cup K_1$ is minor-minimally weakly separating.

\subsection{Nonplanar and nonouter-projective-planar graphs}
In this section, we show that if a graph is nonplanar projective planbar and nonouter-projective-planar, then that graph disjoint union a vertex must be separating. We also show the conditions for that graph to be a minor-minimal graph in regard to being both nonplanar and nonounter-projective-planar.

Recall that a drawing of graph $G$ in the projective plane is a \textit{closed cell embedding} if every face of the graph is bounded by a 0-homologous cycle.

\begin{proposition}\label{addavertex}
Let $G$ be a graph that is nonplanar projective planar and nonouter-projective-planar, then $G \dot\cup K_1$ is a separating projective planar graph.
\end{proposition}

\begin{proof}
All the projective planar embeddings of $G$ are closed cell embeddings, by Corollary \ref{closed cell}. This means all the faces of $G$ are bounded by a 0-homologous cycle. If you add a vertex to one of the faces, there will be a vertex inside and outside of the bounding 0-homologous cycle, because the graph is nonouter-projective-planar.
\end{proof}

In Section 5.5 of this paper, we have identified some members of the set of minor-minimal nonplanar and nonouter-projective-planar graphs. From this, we can conclude that $(K_6-2e) \dot\cup K_1$, $\epsilon_4 \dot\cup K_1$, $\epsilon_6 \dot\cup K_1$, $\kappa_1 \dot\cup K_1$, $LU_1 \dot\cup K_1$, $LU_2 \dot\cup K_1$, and $LU_3 \dot\cup K_1$ are separating projective planar graphs. The graphs $LU_i$ will be defined later, in Section 5.5

\begin{proposition}
The graphs $\kappa_1 \dot\cup K_1$, $\epsilon_4 \dot\cup K_1$, and $\epsilon_6 \dot\cup K_1$ are minor-minimal separating projective planar graphs.
\end{proposition}

\begin{proof}
By Proposition \ref{addavertex}, these graphs must be separating projective planar. Now, we will show they are minor-minimal in this regard. Without loss of generality, consider $\kappa_1 \dot\cup K_1$. Call a minor of this graph $H$.
There are two cases to consider: if $H$ results from deleting the $K_1$ or if $H$ results from taking a minor of $\kappa_1$. First, suppose we have deleted $K_1$. We know $\kappa_1$ is strongly nonseparating, so this minor is not separating. The graph $\kappa_1$ is minor-minimal nonouter-projective-planar, so any minor of it will be outer-projective-planar. Thus, if we embed this minor as its outer-projective-planar drawing, and we embed $K_1$ in the face where all the vertices are in the boundary, this will create a nonseparating drawing. Thus, we can conclude that $\kappa_1 \dot\cup K_1$ is minor-minimal separating projective planar, and by extension, so are $\epsilon_4 \dot\cup K_1$ and $\epsilon_6 \dot\cup K_1$.
\end{proof}

\subsection{Multipartite graphs}
In this section, we explore which complete multipartite graphs are nonseparating.

\begin{proposition}
The nonseparating complete multipartite graphs are exactly $K_{2,2}$, $K_{2,3}$, $K_{2,4}$, $K_{3,3}$, $K_{3,4}$, and $K_{1,n}$ where $n$ is a positive integer.
\end{proposition}

\begin{proof}
Consider $K_{2,n}$. By Proposition 4.9, it is nonseparating if and only if $n \leq 4$. The graphs $K_{3,3}$ and $K_{3,4}$ are nonseparating, as seen in Figures \ref{k33} and \ref{k34}. Consider $K_{4,n}$ and note that $K_{4,4}$ is not projective planar. Note also that every graph $K_{1,n}$ is nonseparating and that every graph $K_{m,n}$ where $m \geq 4$, $n \geq 2$ is separating since $K_{2,5} = K_{5,2}$ is separating.
\end{proof}


\begin{figure}[H]
\centering 
\begin{minipage}[b]{0.4\linewidth}
\centering
\includegraphics[scale=0.35]{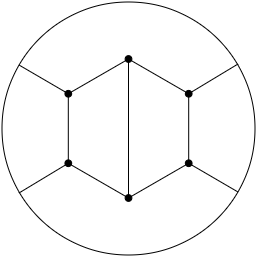}
    \caption{This graph is $K_{3,3}$}
    \label{k33}
\end{minipage}
\quad
\begin{minipage}[b]{0.33\linewidth}
\centering
\includegraphics[scale=0.15]{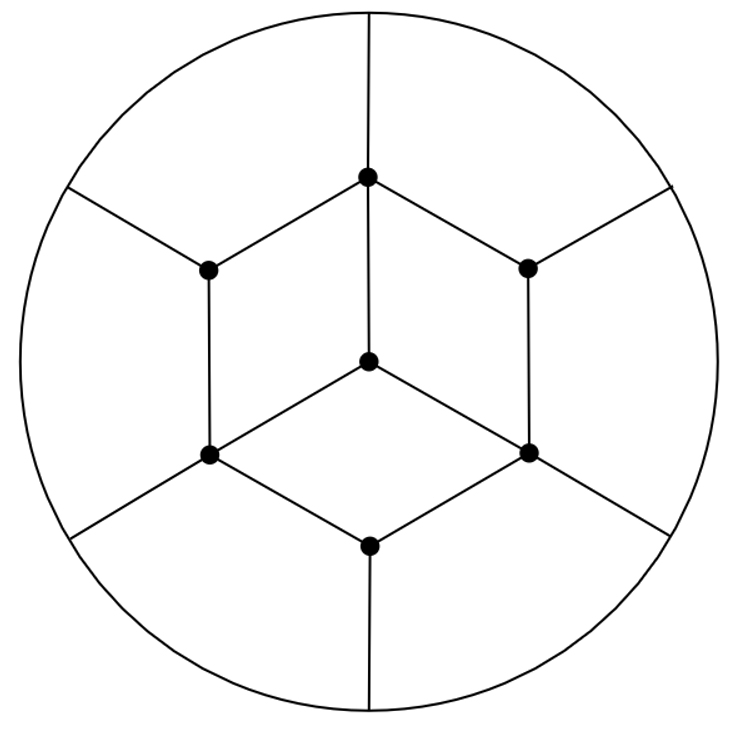}
    \caption{This graph is $K_{3,4}$}
    \label{k34}
\end{minipage}
\end{figure}

\subsection{Conclusion}

This table shows a summary of our results on the 32 minor-minimal nonouter-projective-planar graphs:
\begin{table}[H]
\centering
\resizebox{\columnwidth}{!}{
\begin{tabular}{|c|c|c|c|c|}
\hline
Family & SPP? & SNSPP? & If not SPP, what related graph is weakly SPP?\\
\hline
$\alpha$ & Yes & No & -  \\
\hline
$\beta$ & No & No  & $\beta \dot\cup K_1$ \\
\hline
$\epsilon$ & No & Yes  & $\epsilon \dot\cup K_1$ \\
\hline
$\delta$ &   Yes & No  & - \\
\hline
$\gamma$ & No, except $\gamma_6$ & Yes, except $\gamma_6$ & $\gamma \dot\cup K_1$\\
\hline
$\zeta$ & No, except $\zeta_3$ & Yes, except $\zeta_3$ & $\zeta \dot\cup K_1$\\
\hline
$\eta_1$ & Yes & No & - \\
\hline
$\theta_1$ & Yes & No & -\\
\hline
$\kappa_1$ & No & Yes & $\kappa \dot\cup K_1$\\

\hline
\end{tabular}
}
\caption{Status of minor-minimal nonouter-projective-planar graphs, with respect to being separating projective planar (SPP) or strongly nonseparating projective planar (SNSPP), as well as weakly SPP.}
\label{tab:my_label}
\end{table}

If a minor-minimal nonouter-projective-planar graph is nonseparating, there must be a graph that contains the graph as a minor that is minor-minimal separating projective planar. For example, $\kappa_1 \dot\cup K_1$, $\epsilon_4 \dot\cup K_1$, and $\epsilon_6 \dot\cup K_1$ are minor-minimal separating projective planar graphs, though we don't know if they are minor-minimal. By Proposition \ref{addavertex} and Section 5.5, we know $(K_6-2e) \dot\cup K_1$, $LU_1 \dot\cup K_1$, $LU_2 \dot\cup K_1$, and $LU_3 \dot\cup K_1$ are separating.

\begin{theo}
The following are minor-minimal separating projective planar graphs: graphs in the $\alpha$ family, graphs in the $\delta$ family, $\gamma_6$, $\zeta_3$, $\eta_1$, and $\theta_1$.
\end{theo}

\begin{theo}
The following are minor-minimal weakly separating graphs: graphs in the $\beta$ family, $\epsilon_1 \dot\cup K_1$, $\epsilon_2 \dot\cup K_1$, $\epsilon_3 \dot\cup K_1$, $\epsilon_4 \dot\cup K_1$, $\epsilon_5 \dot\cup K_1$, $\epsilon_6 \dot\cup K_1$, $\gamma_1 \dot\cup K_1$, $\gamma_2 \dot\cup K_1$, $\gamma_3 \dot\cup K_1$, $\gamma_4 \dot\cup K_1$, $\gamma_5 \dot\cup K_1$, $\zeta_1 \dot\cup K_1$, $\zeta_2 \dot\cup K_1$, $\zeta_4 \dot\cup K_1$, $\zeta_5 \dot\cup K_1$, $\zeta_6 \dot\cup K_1$, and $\kappa_1 \dot\cup K_1$
\end{theo}

Dehkordi and Farr \cite{Dehkordi} characterized the set of nonseparating planar graphs as graphs that are outerplanar or a subgraph of wheel graphs, or a subgraph of elongated triangular prism graphs \cite{Dehkordi}. We have extended this research to the projective plane. In the plane, if a graph is nonseparating, it is also strongly nonseparating, so Dehkordi and Farr only have one theorem on the topic. For the projective plane, we have two theorems, though we are only able to characterize some such graphs, at this point.

\begin{theo}
The set of nonseparating projective planar graphs includes the following:
\begin{itemize}
    \item outer-projective-planar graphs
    \item subgraphs of wheel graphs
    \item subgraphs of elongated prism graphs
    \item the $\beta$ family, the $\epsilon$ family, the $\gamma$ family except $\gamma_6$, the $\zeta$ family except $\zeta_3$, and $\kappa_1$
\end{itemize}
\end{theo}

\begin{theo}
The set of strongly nonseparating projective planar graphs includes the following:
\begin{itemize}
    \item outer-projective-planar graphs
    \item subgraphs of wheel graphs
    \item subgraphs of elongated prism graphs
    \item the $\epsilon$ family, the $\gamma$ family except $\gamma_6$, the $\zeta$ family except $\zeta_3$, and $\kappa_1$
\end{itemize}
\end{theo}

\section{3-linked graphs}

A \textit{split 3-link} is a 3-link embedded in the plane with two pieces of the 3-link contained within an embedded $S^1$ with the third piece on the other side of the $S^1$. If there exists no such $S^1$, then the link is a \textit{nonsplit 3-link}. A graph, $G$, is \textit{intrinsically type I 3-linked (II3L)} if every embedding of $G$ in the plane contains a nonsplit type I 3-link. Burkhart et al found three minor-minimal graphs in this set.

\begin{proposition}[Burkhart et al \cite{Burkhart}]\label{burkhart}
The graphs $K_4 \dot\cup K_4$, $K_4 \dot\cup K_{3,2}$, and $K_{3,2} \dot\cup K_{3,2}$ are II3L.
\end{proposition}

They conjectured that this is the complete minor-minimal set. We have used their research as a foundation to explore the set of minor-minimal 3-linked graphs in the projective plane.

A \textit{projective planar 3-link} is a disjoint collection of $3-m$ $S^1$'s and $m$  $S^0$'s, embedded into the projective plane, where $m\in\{1,2\}$. If $m=1$, this is a type I 3-link. If $m=2$, this is a type II 3-link. A \textit{split projective planar 3-link} is a 3-link embedded in the projective plane with two pieces of the 3-link contained within an embedded $S^1$ with the third piece on the other side of the $S^1$. A \textit{nonsplit projective planar 3-link} is when such an $S^1$ does exist. In the figures below, there are only two cases of type I nonsplit 3-links, which are labeled type Ia and type Ib.
    
\begin{figure}[H]
\centering
\includegraphics[scale=0.23]{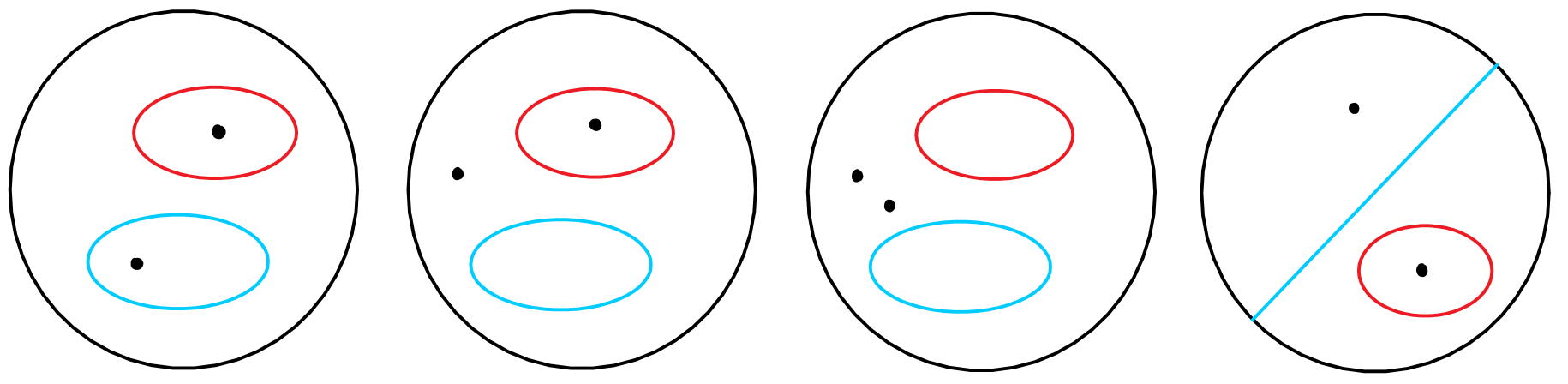}
\caption{A type I 3-link with two $S^1$ and an $S^0$. The embedding on the far left is nonsplit, and the rest are split. These embeddings are type Ia, as neither $S^1$ lies in a disk bounded by the other.}
\label{fig:typeIa}
\end{figure}

\begin{figure}[H]
\centering
\includegraphics[scale=0.31]{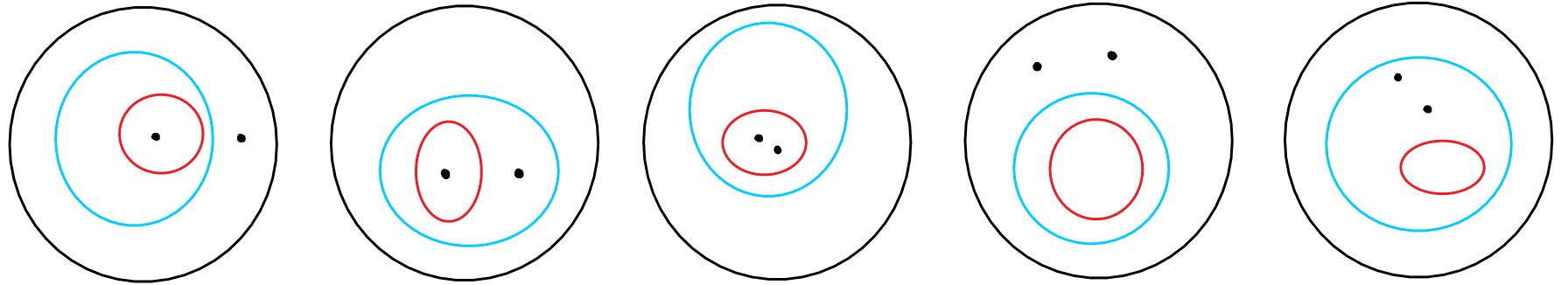}
\caption{A type I 3-link with two $S^1$ and an $S^0$. The embedding on the far left is nonsplit, and the rest are split. These embeddings are type Ib, as one $S^1$ lies in a disk bounded by the other.}
\label{fig:typeIb}
\end{figure}
    
\begin{figure}[H]
\centering
\includegraphics[scale=0.27]{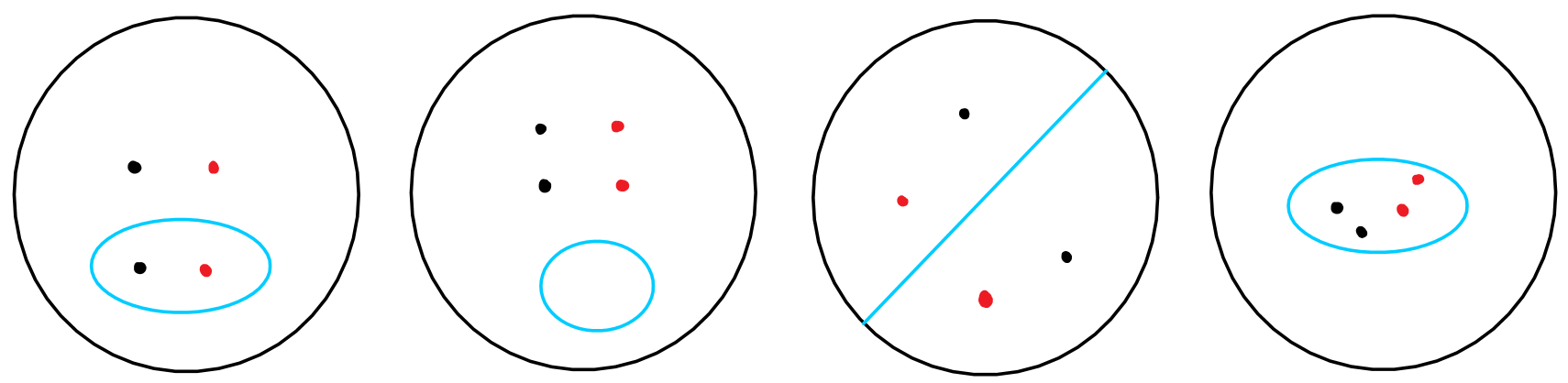}
\caption{A type II 3-link with two $S^0$ and an $S^1$. The embedding on the left is nonsplit, and the embedding in the middle and on the right are split.}
\label{fig:1S1-2S0}
\end{figure}

A projective planar graph, $G$, is \textit{intrinsically projective planar type I 3-linked (IPPI3L)} if every embedding of $G$ in the projective plane has a type I projective planar 3-link. A graph, $G$, is \textit{intrinsically projective planar type II 3-linked (IPPII3L)} if every embedding of $G$ in the projective plane has a type II projective planar 3-link. 

Suppose $G$ is an embedded graph with a 0-homologous cycle. That cycle is a \textit{weak separating cycle} if every vertex of the graph that is not in the cycle is in the interior of the cycle.

If a drawing of graph $G$ in the projective plane is a \textit{closed cell embedding}, that means every face of the graph can be bounded by a 0-homologous cycle. If a graph $G$ is \textit{closed nonseparating}, that means any nonseparating embedding of $G$ in the projective plane is a closed cell embedding. An immediate result of this definition is the following. 

\begin{proposition}
If $G$ closed nonseparating and nonouter-projective planar, then $G\dot\cup K_1$ is separating. 
\end{proposition}

\begin{proof}
Embed $G \dot\cup K_1$ in the projective plane. This drawing is called $D$. Remove the $K_1$, creating a of drawing of $G$, $D_1$. This drawing has two cases. If $D_1$ is separating, then $D$ is also separating. If $D_1$ is not separating, then the drawing is a closed cell embedding. Since it is nonouter-projective planar, no matter what face we embed $K_1$ into, there exists a cycle $C$ such that $C$ bounds the face and there exists a vertex $v$ outside of $C$. Thus, drawing $D$ must also be separating in this case.
\end{proof}

\subsection{IPPI3L graphs with three components}

\begin{proposition}\label{threecomponent}
There are four minor-minimal IPPI3L graphs made of three components. These graphs are $K_4\dot\cup K_4 \dot\cup K_4$, $K_4\dot\cup K_{4} \dot\cup K_{3,2}$, $K_4 \dot\cup K_{3,2} \dot\cup K_{3,2}$, and $K_{3,2} \dot\cup K_{3,2} \dot\cup K_{3,2}$.
\end{proposition}

\begin{proof}
First, consider the graph $K_4\dot\cup K_4 \dot\cup K_4$. We embed these components into the projective plane. If two $K_4$ components are embedded in the projective plane with 1-homologous cycles, they would intersect. So the two cases for the projective planar embedding are if all three $K_4$ components are embedded with all 0-homologous cycles, or two components are embedded with all 0-homologous cycles and the other is embedded with a 1-homologous cycle. When all three $K_4$ components are embedded with 0-homologous cycles, by Theorem \ref{Theorem1}, this is equivalent to if they were all affine embedded. Similarly, if one of the $K_4$ components is embedded with a 1-homologous cycle, the other two components can be deformed into affine embedded graphs. It has been proven that $K_4 \dot\cup K_4$ is II3L, by Proposition \ref{burkhart}. Thus, this embedding also has a nonsplit type I 3-link, which means $K_4\dot\cup K_4 \dot\cup K_4$ is intrinsically projective planar type I 3-linked.

Now, we will verify that $K_4\dot\cup K_4 \dot\cup K_4$ is minor-minimal with respect to being an intrinsically projective planar type I 3-linked graph. Without loss of generality, take a minor that consists of two $K_4$ components and a component that is a minor of $K_4$, which is outerplanar. Embed one $K_4$ with a 1-homologous cycle. Up to equivalence, there is only one such embedding. Now affine embed the other two components. The component that is a minor of $K_4$ is outerplanar. If a component is outerplanar, it does not contain a cycle with a vertex in its interior. This embedding does not contain a nonsplit type I 3-link. Therefore, $K_4\dot\cup K_4 \dot\cup K_4$ is minor-minimal with respect to being an intrinsically projective planar type I 3-linked graph.

Next, consider the graphs $K_4\dot\cup K_4 \dot\cup K_{3,2}$, $K_4 \dot\cup K_{3,2} \dot\cup K_{3,2}$, and $K_{3,2} \dot\cup K_{3,2} \dot\cup K_{3,2}$. These cases follow the same logic as the first case, because $K_{3,2}$ and $K_4$ are both minor-minimal nonouter-projective-planar graphs.
\end{proof}

\begin{proposition}
There is no minor-minimal IPPI3L graph with four or more components.
\end{proposition}

\begin{proof}
Suppose $G$ is a minor-minimal IPPI3L graph with exactly four components. Since Proposition \ref{threecomponent} describes minor-minimal graphs on three components, we know that at most two of the four components of $G$ have $K_4$ or $K_{3,2}$ as a minor. By Theorem \ref{halin}, we know that at least two of the components of $G$ will be outerplanar. Suppose the four components of $G$ are labelled $G_1$, $G_2$, $G_3$, and $G_4$. Without loss of generality, suppose $G_3$ and $G_4$ are outerplanar. 

Let $D$ be an arbitrary drawing of $G_1 \dot\cup G_2$. Embed $G_3$ and $G_4$ into any face of $D$ as affine outer-projective-planar drawings, so that neither $G_3$ lies in a face of $G_4$ or vice versa. Call this drawing $D_1$. Since $G$ is IPPI3L, there exists a projective planar type I 3-link in $D_1$. If every vertex or edge of the 3-link is not in $G_3$ and $G_4$, then $G_1 \dot\cup G_2$ is IPPI3L. Therefore, we can delete $G_3 \dot\cup G_4$, which means $G$ is not minor-minimal. 

If there is a vertex or edge of the 3-link in $G_3 \dot\cup G_4$, it is at most a vertex of the $S_0$ -  they are outerplanar graphs, and by the way they were embedded, there cannot be a cycle bounding a vertex. Those two components contain only one vertex of the $S^0$, and $G_1$ or $G_2$ contains for the second vertex.

Suppose we have a Type Ia nonsplit link, and $G_3$ or $G_4$ contains a vertex within an $S^1$. The graph can be re-embedded where the components $G_3$ and $G_4$ are within the cycle with the other $S^0$ vertex. 

Suppose we have a Type Ib link as in Case 2. Without loss of generality, suppose $G_3$ or $G_4$ contains the external vertex. Then, the graph can be re-embedded where the components $G_3$ and $G_4$ are within the disks bounded by the cycles of both $S^1$ pieces. 

Thus, in either case, the graph does not have a nonsplit type I 3-link. Thus, this is a contradiction, which means that $G_3$ and $G_4$ must be disjoint from the 3-link. Thus, $G_1 \dot\cup G_2$ is IPPI3L. Therefore, we can delete $G_3 \dot\cup G_4$, which means $G$ is not minor-minimal. Therefore, there is no minor-minimal IPPI3L graph with four components.

The argument is similar for every $n$ component graph with $n \geq 4$ has 4 components. We can conclude that there is no minor-minimal IPPI3L graph with four or more components.
\end{proof}

\subsection{Planar separating graphs}

The following generalizes part of Proposition \ref{threecomponent}.

\begin{proposition}\label{separatingIPPI3L}
If $G$ and $H$ are separating projective planar graphs, and $G$ is a planar graph, then $G\dot{\cup}H$ is IPPI3L.
\end{proposition}

\begin{proof}
Consider an arbitrary embedding of $G\dot{\cup}H$, creating drawing $D_1$. Since $G$ and $H$ are separating projective planar graphs, there exist two disjoint 0-homologous cycles $C_1$ and $C_2$ that contains two vertices within them, $v_1$ and $v_2$ respectively. There are 2 cases: one cycle is in the interior of the other, or neither cycle contains the other. Consider the first case. Without loss of generality, $C_1$ is in the interior of $C_2$. Because $C_2$ is separating, we know there is a vertex in the exterior of $C_2$. This creates a Type $1a$ 3-link graph. Consider the second case. Since neither is in the interior of the other, we have two cycles with a vertex inside of them. This would be a Type $1b$ 3-link graph. Since we chose an arbitrary embedding of $G\dot{\cup}H$, this means we can conclude $G\dot{\cup}H$ is IPPI3L.
\end{proof}

\begin{con}\label{IPPI3l-planars}
Suppose $G$ and $H$ are minor-minimal separating projective planar graphs that have planar embeddings. Also, suppose every minor of $G$ and every minor of $H$ is not closed nonseparating. Then $G\dot{\cup}H$ is minor-minimal IPPI3L if both $G$ and $H$ are not II3L. 
\end{con}

An example of this is seen in Figure \ref{fig:thetagamma}.

\begin{figure}[H]
    \centering
    \includegraphics[scale=0.27]{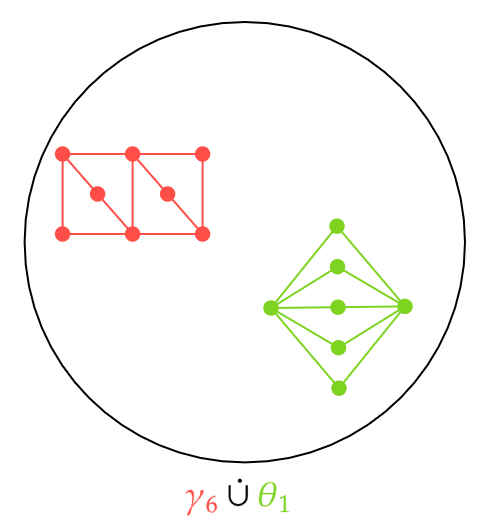}
    \caption{Graph $\gamma_6 \dot\cup \theta_1$ is IPPI3L, as both components are separating projective planar graphs.}
    \label{fig:thetagamma}
\end{figure}

\subsection{Closed nonseparating graphs}

Throughout this section, we assume $G$ and $H$ are projective planar graphs.

\begin{proposition}\label{closed-nonseparating}
Suppose $G$ and $H$ are nonouter-projective-planar. If $G$ and $H$ are closed nonseparating, then $G\dot\cup H$ is IPPI3L.
\end{proposition}

\begin{proof}
If we embed both $G$ and $H$ with affine embeddings, this will create a nonsplit type I 3-link, as both are nonouterplanar. So, without loss of generality, consider a drawing $D_1$ of $G$ with at least one 1-homologous cycle. If $D_1$ is a separating drawing, it has a separating cycle $C$. If we embed $H$ outside of $C$, this creates two disjoint separating cycles, which is a Type $1a$ 3-link. If we embed $H$ inside the cycle $C$, this will create a Type $1b$ 3-link. Now, suppose we have a drawing $D_2$ of $G$ that is nonseparating. This means that it is a closed cell embedding. No matter which face we embed $H$ into, it will produce a Type $1b$ 3-link, as $G$ is nonouterplanar so there will be a vertex outside of the face that $H$ is embedded into.
\end{proof}

\begin{proposition}\label{minorminimalclosed}
Suppose $G$ and $H$ are minor-minimal nonouter-projective-planar, where both of them are planar. Also, $G$ and $H$ are not II3L. If $G$ and $H$ are closed nonseparating, then $G\dot\cup H$ is minor-minimally IPPI3L.
\end{proposition}

\begin{proof}
Suppose $G$ and $H$ are minor-minimal nonouter-projective-planar and closed nonseparating. Also, suppose $G$ and $H$ are not II3L. We have proven in Proposition \ref{closed-nonseparating} that $G\dot{\cup} H$ is IPPI3L. Now, we will prove that $G\dot{\cup} H$ is minor-minimal under this property. Without loss of generality, suppose $L$ is a minor of $H$. This means $L$ is outer-projective-planar. Embed $L$ as an outer-projective drawing, called $D_1$. All vertices of $L$ are contained in the boundary of one face, $F$. If we embed $G$ into $F$, there is no 3-link, since $G$ is not II3L.
\end{proof}

\begin{con}\label{ep,zeta}
The graphs $\epsilon_1$ and all members of the $\zeta$ family except $\zeta_3$ are closed nonseparating.
\end{con}

By Proposition \ref{minorminimalclosed} and Conjecture \ref{ep,zeta}, there are many graphs that we conjecture to be minor-minimal IPPI3L. An example of this would be $\epsilon_1 \dot\cup \epsilon_1$, as shown in Figure \ref{fig:epsilon}.

\begin{figure}[H]
    \centering
    \includegraphics[scale=0.28]{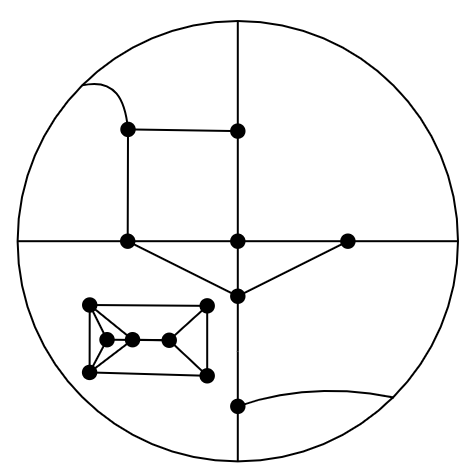}
    \caption{This graph is $\epsilon_1 \dot\cup \epsilon_1$}
    \label{fig:epsilon}
\end{figure}

\subsection{Graphs glued at a vertex}
Suppose $D$ is a planar drawing of a graph $G$. A vertex of $G$ is a \textit{separating planar vertex in $D$} if it is contained in all separating cycles of $D$ and we will say that a vertex of $G$ is a \textit{separating planar vertex} if it is contained in all separating cycles for all planar drawings of $G$. The graphs $\eta_1$ and $\zeta_3$ have no separating planar vertices. In the graph $\theta_1$, the two separating planar vertices are the degree 5 vertices. For $\epsilon_1$, $\gamma_6$, $\delta_1$, and $\delta_2$ the separating planar vertices are the degree 4 vertices.

\begin{proposition}\label{gluedatzeta}
Suppose $G$ is a separating projective planar graph that is planar and is not II3L. Then $G$ glued at any vertex called $v$ to $\zeta_3$ is IPPI3L.
\end{proposition}

\begin{proof}
The graph $\zeta_3$ has only two unique projective planar embeddings, as seen in Figures \ref{Zeta3Sep}. No matter how $\zeta_3$ is embedded, there is a separating cycle that does not go through $v$. Consider an arbitrary embedding of $\zeta_3$. Since every vertex is equivalent, without loss of generality, we pick vertex $v$ to glue to $G$. Now we must embed $G$ into the projective plane as well. Since $G$ is separating, it will contain a separating cycle, no matter how it is embedded. Even if the separating cycle of $G$ goes through vertex $v$, there exists a disjoint cycle in $\zeta_3$ that does not go through vertex $v$. Thus, it has a type 1a 3-link. If $\zeta_3$ is affine and embedded within a separating cycle of $G$, this creates a type 1b 3-link. We can conclude that every embedding with $G$ and $\zeta_3$ glued at a vertex will be IPPI3L.
\end{proof}

\begin{proposition} \label{minorminimalzetaglued}
Suppose $G$ is a minor-minimal nonouter-projective-planar graph. Also suppose $G$ is connected, separating projective planar, and planar. The graph $G$ glued at any vertex to $\zeta_3$ is minor-minimal IPPI3L.
\end{proposition}

\begin{proof}
In Proposition \ref{gluedatzeta}, we proved $G$ glued to $\zeta_3$ at any vertex is IPPI3L. Now, we will prove that it is minor-minimal in regards to this property.

Suppose $\zeta_3$ and $G$ are glued by a vertex $v$. Consider a minor of this resulting graph, called $H$.
 
First, consider the case where $H$ does not contain vertex $v$ because $v$ has been deleted. Graph $H$ has at least two components. Consider the part of $H$ that is a proper minor of either $G$ or $\zeta_3$ that is outer-projective-planar, and call it $H_1$. Let $D_1$ be an outer-projective-planar drawing of $H_1$.Then there exists a face $F_1$ of $D_1$ such that all the vertices of $H_1$ are in its boundary. Embed the other pieces of $H$ in $F_1$. This drawing is not projective-planar type I 3-linked.

Second, consider the case where $H$ does contain vertex $v$. Consider the part of $H$ that is a proper minor of either $G$ or $\zeta_3$ that is outer-projective-planar, and call it $H_2$. Let $D_2$ be an outer-projective-planar drawing of $H_2$. Then there exists a face $F_2$ of $D_2$ such that all the vertices of $H_2$ are in its boundary. Embed the other pieces of $H$ in $F_2$. This drawing is not projective-planar type I 3-linked.

We can conclude that $G$ glued to $\zeta_3$ at any vertex is minor-minimal IPPI3L.

\end{proof}

Graph $G$ in this proposition could be:  the $\delta$ family, $\gamma_6$, $\zeta_3$, $\eta_1$, and $\theta_1$

An example of this type of graph would be $\zeta_3$ glued with $\gamma_6$ at any vertex, as shown in Figure \ref{fig:zeta3gamma6}.

\begin{figure}[H]
    \centering
    \includegraphics[scale=0.22]{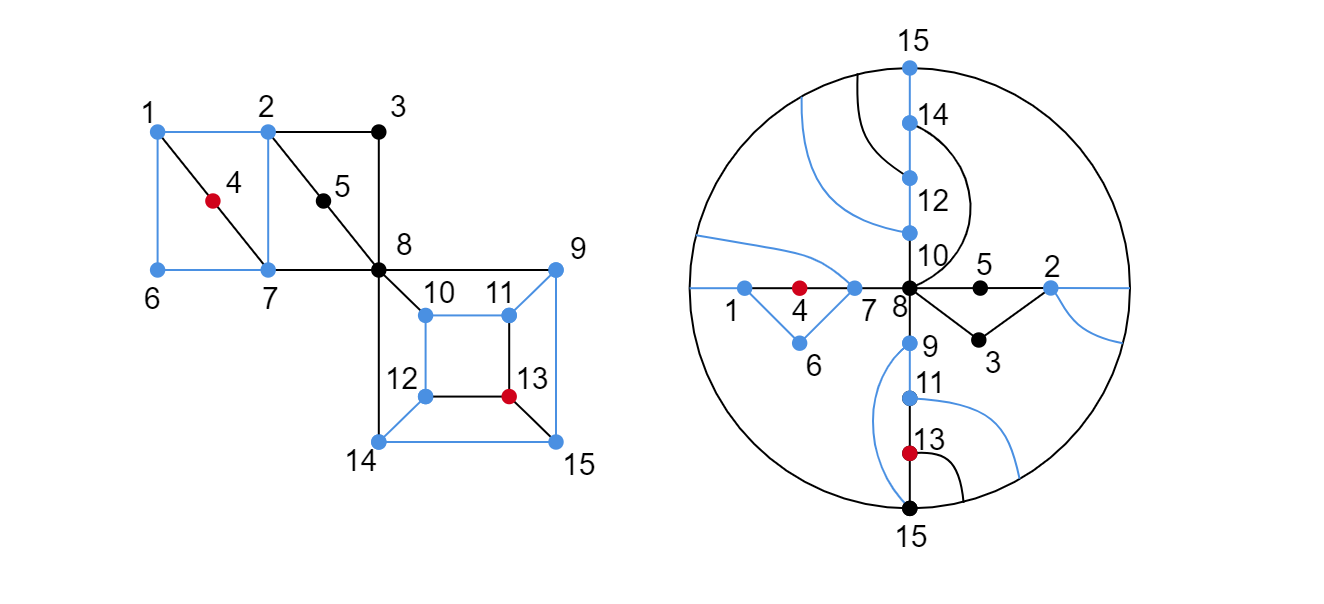}
    \caption{The graph $\zeta_3$ glued with $\gamma_6$ glued at vertex 8. The graph is type 1a linked - the two $S^1$'s are highlighted in blue, and the $S^0$ are highlighted in red.}
    \label{fig:zeta3gamma6}
\end{figure}

\begin{con}
Let the graph $G$ be two copies of $\epsilon_1$ be glued together, where the two vertices that were glued, $v_1$ and $v_2$, are not separating planar vertices. Then $G$ is minor-minimal IPPI3L.
\end{con}

\begin{con}
Let the graphs $G$ and $H$ be closed nonseparating, nonouter-projective-planar, and planar. Suppose the vertex $v_1$ of $G$ and $v_2$ of $H$ are glued together at vertex $v$, where $v_1$ and $v_2$ are not separating planar vertices. The resulting graph will be IPPI3L.
\end{con}

\begin{con}
Let the graphs $G$ and $H$ be closed nonseparating, minor-minimal nonouter-projective-planar, and planar. Suppose the vertex $v_1$ of $G$ and $v_2$ of $H$ are glued together at vertex $v$, where $v_1$ and $v_2$ are not separating planar vertices. The resulting graph will be minor-minimal IPPI3L.
\end{con}

\subsection{Nonplanar and nonouter-projective-planar graphs}

Throughout this section, assume $G$ is a projective planar graph. In this section, we examine graphs that are the disjoint union of two components. One component is a graph that is nonplanar and nonouter-projective-planar, and the second component is a graph that is nonouterplanar. We prove that the union results in an IPPI3L graph. We also show that the property of being planar or outer-projective-planar is a minor closed property. Following that, we give conditions to make graphs that minor-minimal in regards to that property. Finally, we characterize graphs that are minor-minimal nonplanar and nonouter-projective-planar. 

\begin{proposition}
Suppose a graph $G$ has the property that it has a planar embedding or an outer-projective-planar embedding. This property is a minor closed property.
\end{proposition}

\begin{proof}
Follows, since being planar is a minor closed property, as is being outer-projective planar.
\end{proof}

 For a graph $G$ to be nonplanar and nonouter-projective-planar, it must have as minors one minor-minimal nonouter-projective planar graph and one minor-minimal nonplanar graph.

\begin{col}
The set of minor-minimal nonplanar and nonouter-projective-planar graphs is finite. This is because of Robertson and Seymour's Minor Theorem \cite{robertson}.
\end{col}

A \textit{closed cell embedding} is an embedding where every face is bounded by a 0-homologous cycle.

\begin{lem}\label{k33k5}
All projective planar embeddings of $K_5$ and $K_{3,3}$ are closed cell embeddings.
\end{lem}

\begin{figure}[H]
    \centering
    \includegraphics[scale=0.18]{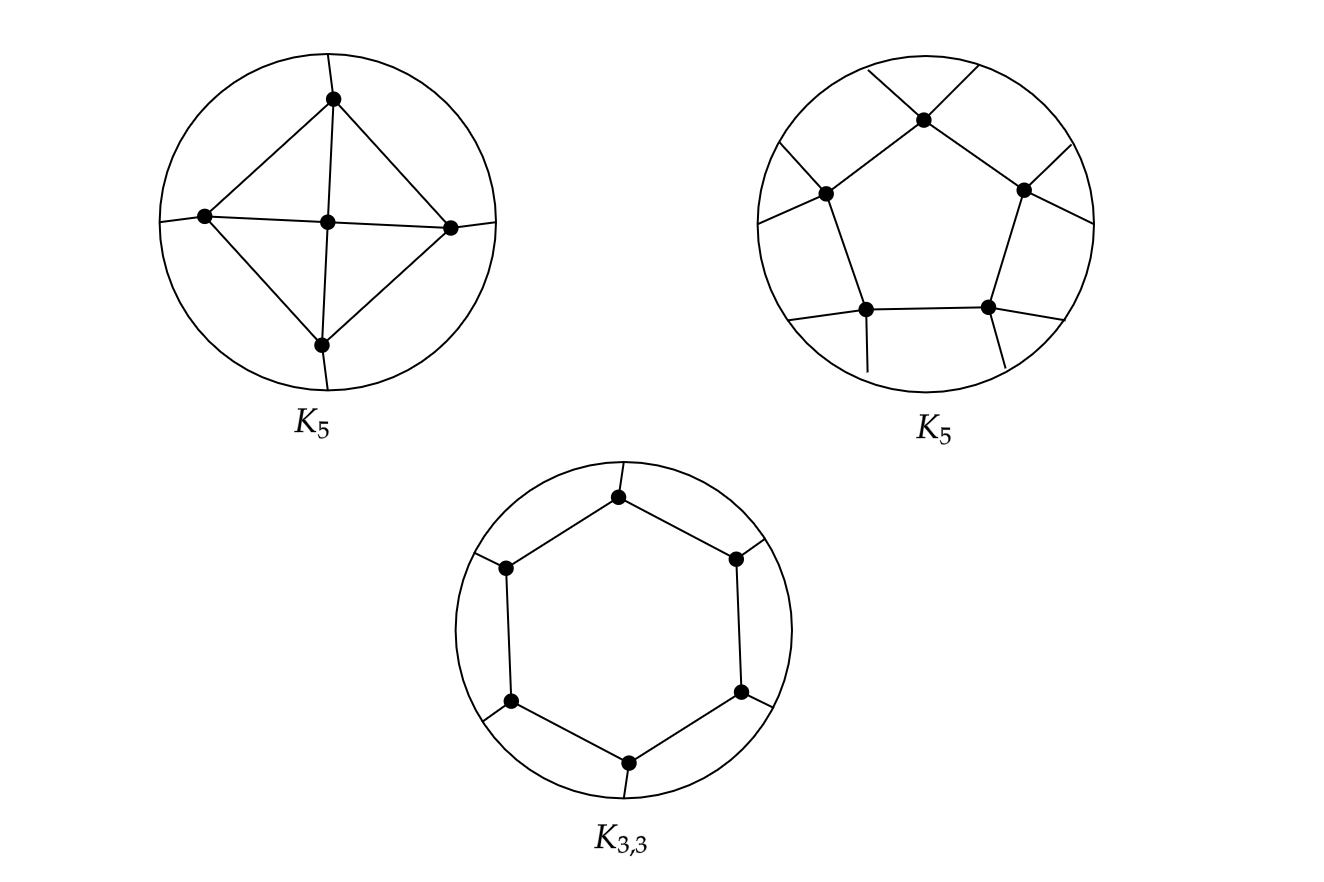}
    \caption{Up to symmetry,these are the only possible embeddings of $K_5$ and $K_{3,3}$ \cite{Maharry}.}
    \label{fig:my_label}
\end{figure}

\begin{proof}
\noindent First consider $K_5$. It follows from Maharry et al that there are exactly two classes of embeddings of (unlabelled) $K_5$ (see Figure 43 of \cite{Maharry}). Next, consider $K_{3,3}$. Maharry et al show every embedding of labelled $K_{3,3}$ in Figure 40 \cite{Maharry}. It appears as if there are many embeddings, but they are all equivalent, ignoring the labels. Thus, there is only embedding of $K_{3,3}$.
\end{proof}

Given a graph $G$ with vertex $v$ in $G$, define a \textit{vertex splitting} of $v$ to be the graph $G'$ obtained from $G$ with vertex set $V(G)-\{v\} \cup\{v_1,v_2\}$ and edge set including $(v_1,v_2)$, and $G$ is the result of contracting $G'$ along $(v_1,v_2)$.
\begin{lem} \label{lemma4lemma}
Let $D$ be a closed cell projective planar drawing of graph $G$. If we split a vertex in that drawing so that the new drawing $D_1$ remains projective planar, then $D_1$ is also closed cell. 
\end{lem}

\begin{proof}
Arbitrarily choose a projective planar drawing $D_1$ of graph $G$. Let $v$ be the vertex in $D_1$ that we are splitting. 

\begin{enumerate}
    \item If $v$ is degree 0, this will create a disjoint edge. This will not effect the faces, so the same 0-homologous cycles that bound the faces of $D$ will also bound the faces of $D_1$. Thus, $D_1$ is a closed cell embedding, which means every projective planar drawing of $K_1$ is a closed cell embedding.
    \item If $v$ is degree 1, this will also not affect the faces, so the same 0-homologous cycles that bound the faces of $D$ will also bound the faces of $D_1$. Thus, $D_1$ is a closed cell embedding, which means every projective planar drawing of $K_1$ is a closed cell embedding.
    \item If $v$ is degree 2 or higher, we arbitrarily split the vertex. An example is seen in Figure \ref{Lemma4Lemmapic}. The splitting only affected the face boundary cycles of two faces of the graph - the rest remained the same. The two faces that were affected remain bounded by 0-homologous cycles, which are one vertex longer. 

\end{enumerate}

\begin{figure}[H]
   \centering
   \includegraphics[scale=0.34]{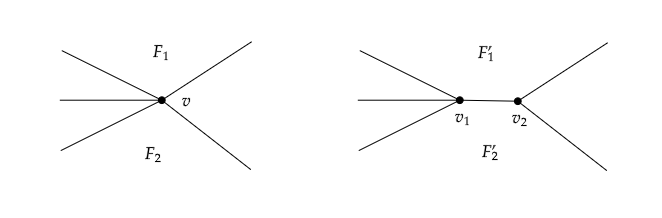}   \caption{Splitting a vertex}
   \label{Lemma4Lemmapic}
\end{figure}
\end{proof}

Let $G$ be a planar graph and $C$ in $G$ is a cycle. Then we will say that we have a path $P$ \textit{glued in $C$} when we connect an $n$-path such that the endpoints are glued to two different vertices in $C$, and the path does not intersect $C$ anywhere else.

\begin{lem}\label{closed cell - majors}
Suppose $H$ is a graph with all its projective planar embeddings are closed cell embeddings. Suppose graph $G$ is a connected and projective planar graph that is created by splitting vertices and adding paths glued in the cycles $C$ of $H$. Then all projective planar embeddings of $G$ are closed cell embeddings.
\end{lem}

\begin{proof}
Suppose $H$ is a graph with all its projective planar embeddings are closed cell embeddings. For case one, suppose we add a path glued in cycle $C$ of $G$, creating $G_1$. Arbitrarily choose a projective planar drawing $D_1$ of $G_1$. Note that $D_1$ is a drawing of $H$ with a path added. Since $H$ has only closed cell embeddings, no matter what face the path is in, it will divide a face that is bounded by a 0-homologous cycle. This results in two faces, which are each bounded by two 0-homologous cycles. Thus, $D_1$ is a closed cell embedding, which means every projective planar drawing of $G_1$ is a closed cell embedding. For case two, suppose we split a vertex in $H$ to create $G_1$. By Lemma \ref{lemma4lemma}, every projective planar drawing of $G_2$ is a closed cell embedding. Therefore, for every $G$ created by gluing paths in cycles of $H$, or splitting vertices in $H$, all the projective planar embeddings of $G$ are closed cell embeddings.
\end{proof}

\begin{col}\label{closed cell}
Suppose $G$ is a nonplanar graph obtained by either splitting vertices or gluing paths in the cycles of $K_{3,3}$ or $K_5$. All projective planar embeddings of $G$ are closed cell embeddings.
\end{col}

\begin{proof}
Suppose a nonplanar graph $G$ has a projective planar drawing $D$. Since $G$ is nonplanar, then it contains $K_{3,3}$ or $K_5$ as a minor. By Lemma \ref{k33k5}, all the projective planar embeddings of $K_{3,3}$ and $K_5$ are closed cell embeddings. By Lemma \ref{closed cell - majors}, all the projective planar embeddings of $G$ are also closed cell embeddings. We can conclude all projective planar embeddings of a nonplanar graph $G$ are closed cell embeddings.
\end{proof}

\begin{proposition}\label{nonplanar}
Let $G$ and $H$ be projective planar. Suppose $G$ is nonplanar and nonouter-projective-planar, and all its planar minors are not II3L. Also, suppose $G$ is obtained by either splitting vertices or gluing paths in the cycles of $K_{3,3}$ or $K_5$. Also suppose $H$ is planar and nonouterplanar. Then $G \dot\cup H$ is IPPI3L.
\end{proposition}

\begin{proof}
Suppose graph $G$ is nonplanar projective planar and nonouter-projective-planar, and $H$ is nonouterplanar. First, embed $G$ into the projective plane. Call the drawing $D$. Then, embed $H$ in any face $F$ of $D$. By Corollary \ref{closed cell}, $D$ must be a closed cell embedding, which means $F$ is bounded by a 0-homologous cycle, $C_1$. Since $G$ is nonouter-projective-planar, then there is a vertex that is not in the closure of $F$. Label this vertex $v_1$. Because $H$ is nonouterplanar, it has $K_4$ or $K_{3,2}$ as a minor. Without loss of generality, suppose it has a $K_4$ subdivision. The vertex that is within the outer cycle of the $K_4$ subdivision and $v_1$ constitute the $S^0$ piece. The cycle $C_1$ and the cycle that bounds $K_4$ are the two $S^1$'s. This means this graph is IPPI3L.

\end{proof}

An example of a graph that would be IPPI3L by Proposition \ref{nonplanar} would be $K_6 \dot\cup K_{3,2}$, because $K_6$ is nonplanar and nonouter-projective-planar, and $K_{3,2}$ is planar and nonouterplanar.

We will define the term \textit{separating cycles} to be 0-homologous cycles in a drawing of a graph that contain at least one vertex in its interior and one vertex on its exterior. The next proof is about graphs that do not have two disjoint separating cycles in every drawing of the graph.

\begin{proposition}\label{minorminIPPI3L}
Let $G$ and $H$ be projective planar. Suppose $G$ is minor-minimal nonplanar and nonouter-projective-planar, and all its planar minors are not II3L. Also, suppose $G$ is obtained by either splitting vertices or gluing paths in the cycles of $K_{3,3}$ or $K_5$. Also suppose $H$ is minor-minimal nonouterplanar. This implies that $G \dot\cup H$ is minor-minimal IPPI3L.
\end{proposition}

\begin{proof}
Suppose $G$ is projective planar, minor-minimal nonplanar, and nonouter-projective-planar. Since we have assumed all of $G$'s planar minors are not II3L, it does not have two disjoint separating cycles. Suppose $H$ is projective planar and minor-minimal nonouterplanar.
In Proposition \ref{nonplanar}, we proved that $G\dot{\cup} H$ is IPPI3L. Now, we will prove that $G\dot{\cup} H$ is minor-minimal under this property.

Suppose we have a minor of $G$, $L$. Since $L$ is a minor, it must be planar or outer-projective-planar. First, suppose $L$ is outer-projective-planar. Then embed $L$ in the projective plane, with an outer projective drawing $D_1$. Then embed $H$ in $D_1$ inside the face that contains all the vertices in its closure. That is 3-linkless. Second, suppose $L$ is planar. Affine embed $L$ in the projective plane. Embed $H$ with a 1-homologous cycle. We know $H$ will not have a cycle bounding a vertex, because $H$ is either $K_4$ or $K_{3,2}$. Also, $L$ can have at most one cycle bounding a vertex, because it is not II3L. Thus, the minor is not IPPI3L.

Next, suppose $J$ is a minor of $H$ that is outerplanar. Embed $G$ in the projective plane. Call the drawing $D_2$. Now embed $J$ in $D_2$. The embedding of $J$ will not have a cycle bounding a vertex, and $G$ can have at most one. Thus, the minor is not IPPI3L. We can conclude that $G\dot{\cup} K$ is minor-minimal IPPI3L.
\end{proof}

Now, let us examine the set of graphs that are minor-minimal nonplanar and nonouter-projective-planar, that do not have two disjoint separating cycles. The graphs $\epsilon_4$, $\epsilon_6$, and $\kappa_1$ have these properties. Now let us consider $K_6$. It is nonplanar and nonouter-projective-planar, but it is not minor-minimal.

\begin{proposition}\label{5.13}
The graph $K_6$ minus two edges, where the edges may or may not be adjacent, is minor-minimal nonplanar and nonouter-projective-planar.
\end{proposition}

\begin{proof}
The two graphs - $K_6 -2e$ where the edges are adjacent or nonadjacent - are both nonouter-projective-planar because they have $\gamma_1$ as a minor. Also, they both are nonplanar because they contain $K_{5}$ as a minor.

Now, consider deleting a third edge. If we delete a three edges that connect to the same vertex, the resulting graph is outer-projective-planar. If we delete two edges that connect to the same vertex and a third that is disjoint, the resulting graph is outer-projective-planar. If we delete three edges that form a path, the resulting graph is planar. If we delete three edges that form a 3-cycle, the resulting graph is outer-projective-planar. If we delete three disjoint edges, the resulting graph is planar. Thus, we cannot subtract a third edge and keep the desired properties.

Because we cannot subtract a third edge, this means we also cannot delete a vertex. We also cannot contract an edge, as this would result in a graph that has only five vertices, $K_5$ has an outer-projective-planar embedding, as shown in Figure \ref{fig:my_label}. Thus, $K_6 -2e$ is minor-minimal nonplanar and nonouter-projective-planar.
\end{proof}
\begin{figure}[H]
\centering
\includegraphics[scale=0.25]{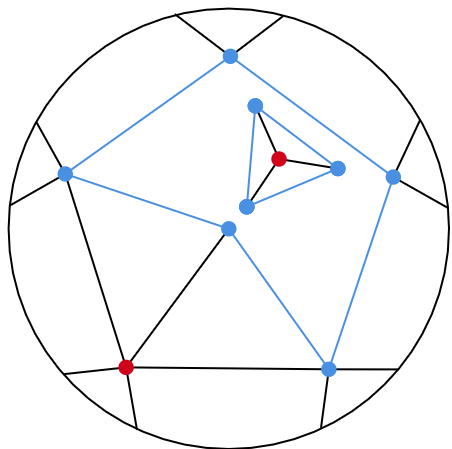}
\caption{This graph is $(K_6-2e) \dot\cup K_4$. The blue cycles are the two $S^1$'s, and the red vertices are the two pieces of the $S^0$.}
\end{figure}

Next, consider the other graphs in the Petersen family. The Petersen family is the set of seven graphs that can be obtained from the Petersen graph by repeated $\Delta - Y$ and $Y- \Delta$ exchanges. The family also includes $K_6$. All other members of the family have $\kappa_1$ as a minor. Therefore, they would not be minor-minimal nonplanar and nonouter-projective-planar.

The graph $\kappa_1$ has $K_{3,3}$ as a minor. Exploring graphs that have $K_5$ as a minor may yield more graphs in the set. 

\begin{figure}[H]
    \centering
    \includegraphics[scale=0.45]{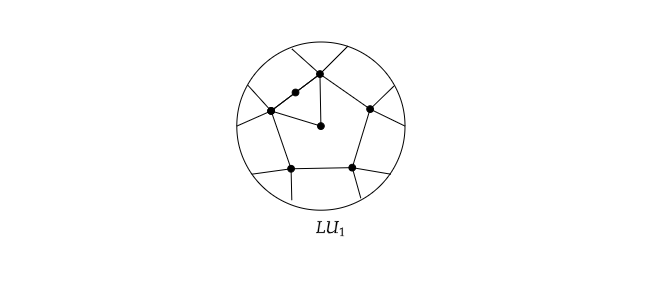}
       \vskip -.4in
    \caption{This graph has both $K_5$ and $\gamma_3$ as minors, so it is minor-minimal nonplanar and nonouter-projective-planar.}
    \label{fig:LU1}
 
\end{figure}

\begin{proposition}\label{5.14}
The graph $LU_1$, pictured in Figure \ref{fig:LU1}, is minor-minimal nonplanar and nonouter-projective-planar.
\end{proposition}

\begin{proof}

Now we want to prove this graph is minor-minimal with regards to being nonplanar and nonouter-projective-planar. There are three types of edges, given different colors in the drawing below. 

\begin{figure}[H]
    \centering
    \includegraphics[scale=0.3]{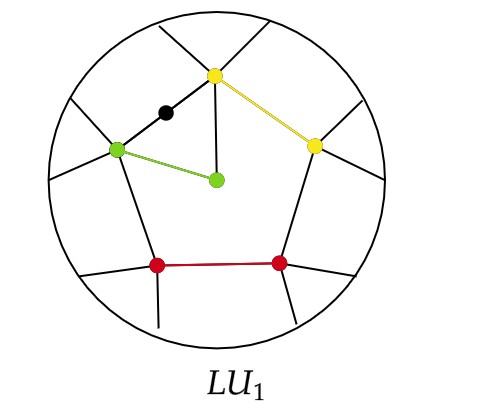}
    \caption{Different types of edges of $LU_1$}
    \label{fig:LU1_color}
\end{figure}

The green edge shows an edge that connects a degree 2 and degree 5 vertex. If we delete the green edge, the graph is outer-projective-planar. The red edge shows an edge that connects two degree 4 vertices. If we delete the red edge, the resulting minor is planar. The yellow edge connects a degree 4 and degree 5 vertex. If we delete the yellow edge, the resulting minor is planar. This means that we also cannot delete vertices, because that would delete edges. Now consider edge contractions. If we contract the green edge, the resulting minor will be outer-projective-planar. If we contract the yellow or red edges, the resulting minor will be planar. Thus we can conclude that this graph is minor-minimal with regards to being nonplanar and nonouter-projective-planar.
\end{proof}

\begin{figure}[H]
    \centering
    \includegraphics[scale=0.32]{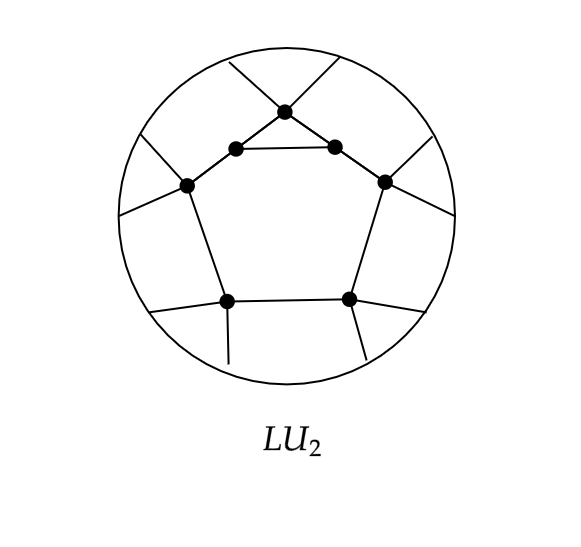}
   \vskip -.4in
   \caption{This graph has both $K_5$ and $\epsilon_1$ as minors, so it is minor-minimal nonplanar and nonouter-projective-planar.}
    \label{fig:LU2}
\end{figure}

\begin{proposition}\label{5.15}
The graph $LU_2$, pictured in Figure \ref{fig:LU2}, is minor-minimal nonplanar and nonouter-projective-planar.
\end{proposition}

\begin{proof}

Now we want to prove this graph is minor-minimal with regards to being nonplanar and nonouter-projective-planar. There are three types of edges, given different colors in the drawing below. 
\begin{figure}[H]
    \centering
    \includegraphics[scale=0.32]{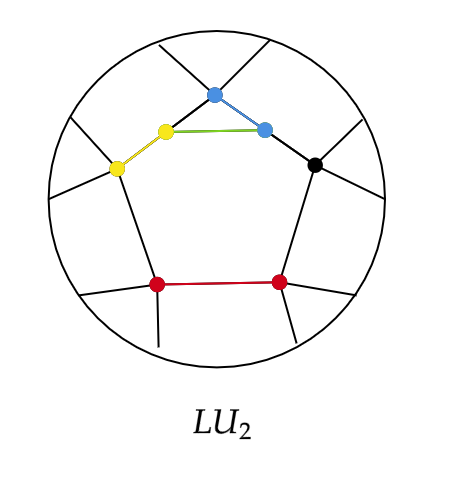}
    \vskip -.25in
    \caption{Different types of edges of $LU_2$.}
    \label{fig:LU2colored}
\end{figure}
The green edge connects two degree 3 vertices. If we delete the green edge, the resulting minor will be outer-projective-planar. The red edge connects two degree 4 vertices. If we delete the red edge, the resulting minor will be planar. The yellow edge connects a degree 3 and a degree 4 vertex. If we delete the yellow edge, the resulting minor will be outer-projective-planar. The blue edge also connects a degree 3 and a degree 4 vertex, where the degree 4 vertex is adjacent to the other degree 3 vertex. If we delete the blue edge, the resulting minor is outer-projective-planar. This means we also cannot delete a vertex, as this will delete an edge. If we contract the green edge, the resulting graph will be outer-projective-planar. If we contract the red edge, the resulting minor is planar. If we contract the yellow edge, the resulting minor is outer-projective-planar. If we contract the blue edge, the resulting minor is outer-projective-planar. Thus, we can conclude this graph is minor-minimal nonplanar and nonouter-projective-planar.
\end{proof}

\begin{figure}[H]
    \centering
    \includegraphics[scale=0.63]{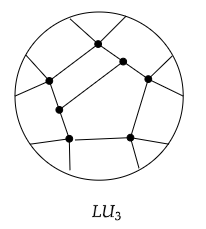}
    \caption{A graph that contains both $K_5$ and $\gamma_1$ as minors, so it is nonplanar and nonouter-projective-planar.}
    \label{fig:LU3}
\end{figure}

\begin{proposition}\label{5.16}
The graph $LU_3$, pictured in Figure \ref{fig:LU3}, is minor-minimally nonplanar and nonouter-projective-planar.

\end{proposition}

\begin{proof}
Now we want to prove this graph is minor-minimal with regards to being nonplanar and nonouter-projective-planar. There are three types of edges, given different colors in the drawing below. 
\begin{figure}[H]
    \centering
    \includegraphics[scale=0.63]{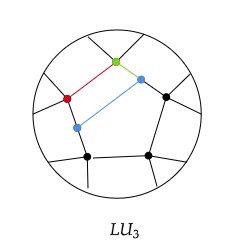}
    \caption{Different types of edges of $LU_3$.}
    \label{fig:LU3colored}
\end{figure}
The green edge connects a degree 3 vertex and a degree 4 vertex. If we delete the green edge, the resulting minor will be planar. The red edge connects two degree 4 vertices. If we delete the red edge, the resulting minor will be outer-projective-planar. The blue edge connects twp degree 3 vertices. If we delete the blue edge, the resulting minor will be outer-projective-planar. This means we also cannot delete a vertex, as this will delete an edge.

If we contract the green edge, the resulting graph will be outer-projective-planar. If we contract the red edge, the resulting minor is planar. If we contract the blue edge, the resulting minor is planar. Thus, we can conclude this graph is minor-minimal nonplanar and nonouter-projective-planar.
\end{proof}

\begin{proposition} Any graph of the form $G \dot\cup H$, where $G$ is minor-minimal nonouter-projective-planar and nonplanar, $G$ is not II3L, and $H$ is minor-minimal nonouterplanar, is minor-minimal IPPI3L. That is, $H$ is either $K_4$ or $K_{3,2}$, and $G$ is either $K_6-2e$, $\epsilon_4$, $\epsilon_6$, $\kappa_1$, $LU_1$, $LU_2$, $LU_3$, or some other minor-minimal nonouter-projective-planar and nonplanar graph that is not II3L.

\end{proposition}

\begin{proof}
This directly follows from Propositions \ref{minorminIPPI3L}, \ref{5.13}, \ref{5.14}, \ref{5.15}, and \ref{5.16}.
\end{proof}

\subsection{Summary and Open Questions}

\begin{proposition}
The following graphs are minor-minimal IPPI3L: 
\begin{itemize}
    \item $K_4\dot\cup K_4 \dot\cup K_4$, $K_4\dot\cup K_{4} \dot\cup K_{3,2}$, $K_4 \dot\cup K_{3,2} \dot\cup K_{3,2}$ or $K_{3,2} \dot\cup K_{3,2} \dot\cup K_{3,2}$ (Proposition \ref{threecomponent}).
    \item $G \dot\cup H$, where $G$ and $H$ are minor-minimal nonouter-projective-planar, not II3L, and closed nonseparating. (Proposition \ref{minorminimalclosed}).
    \item $G$ glued at any vertex to $\zeta_3$, where $G$ is a minor-minimal nonouter-projective-planar, separating projective planar, and planar (Proposition \ref{minorminimalzetaglued}).
    \item $G\dot{\cup} H$, where $G$ is minor-minimal nonplanar and nonouter-projective-planar, and not II3L, and $H$ is minor-minimal nonouterplanar. This includes $(K_6-2e) \dot\cup H$, $\epsilon_4 \dot\cup H$, $\epsilon_6 \dot\cup H$, $\kappa_1 \dot\cup H$, $LU_1 \dot\cup H$, $LU_2 \dot\cup H$, and $LU_3 \dot\cup H$, where $H$ is either $K_4$ or $K_{3,2}$ (Proposition 5.17).
    
\end{itemize}
\end{proposition}

There are open questions that remain unexplored. First, what is the complete set of minor-minimal IPPI3L graphs? We have discovered many graphs, but the set has not been completed. One potential direction would be to examine minor-minimal separating graphs that are glued at an edge rather than a vertex. Secondly, an open question would be what characterizes IPPII3L graphs, especially minor-minimal ones? Also, what is the complete set of graphs that are minor-minimal nonplanar and nonouter-projective-planar? Finally, what is the complete set of graphs that are closed nonseparating?

\section{Acknowledgements}
This material is based upon work obtained by the research group at the 2021 Research Experience for Undergraduates Program at SUNY Potsdam and Clarkson University, advised by Joel Foisy and supported by the National Science Foundation under Grant No. H98230-21-1-0336; St. Catharine's College, University of Cambridge; and Universidad Autónoma del Estado de Hidalgo.

\newpage






\begin{thebibliography}{9}
\bibitem{Archdeacon}
Archdeacon, D.; Hartsfield, N.; Little, C.H.H.; Mohar, B.,  \textit{Obstruction sets for outer-projective-planar graphs}. Ars Combinitoria. Vol 49 (1998), 113-127.

\bibitem{Burkhart}
Burkhart, M.; Castillo, A.; Doane, J.; Foisy, J.; Negron, C., \textit{Intrinsically spherical 3-linked graphs}, arXiv, 2107.08953, math.CO, 2021.

\bibitem{CH} Chartrand, G.; Harary, F., \textit{Planar permutation graphs.}  Ann. Inst. H. Poincar\'e Sect. B (N.S.) 3 (1967), 433-438.

\bibitem{Dehkordi}
Dehkordi, H.; Farr, G., \textit{Nonseparating planar graphs}. The Electronic Journal of Combinatorics. Vol 28, No. 1. Jan 2021.

\bibitem{Diestel}
Diestel R, \textit{Graph theory}. Springer-Verlag, GTM 173, 5th edition 2016/17.

\bibitem{REU2021IPL}
Foisy, J., Knowles, E.; Nolasco Hernandez, U; Shen, Y; Topete Galván, L; Wickham, L., \textit{Intrinsically projectively linked graphs}, preprint.

\bibitem{Glover}
Glover, H.; Huneke, J.; San Wang, C., \textit{103 Graphs that are irreducible for the projective plane}. Journal of Combinatorial Theory, Series B 27 (1979), 332-370.

\bibitem{Halin}
Halin, R., Über einen graphentheoretischen Basisbegriff und seine Anwendung auf Färbensprobleme, Doctoral thesis, Köln, 1962.



\bibitem{Maharry}
Maharry, J.; Robertson, N.; Sivaraman, V.; Slilaty, D., \textit{Flexibility of projective-planar embeddings}. Journal of Combinatorial Theory, Series B 122 (2017), 241-300.

\bibitem{robertson}
Robertson, N., Seymour, P.D., \textit {Graph minors. XX. Wagner's conjecture}, J. Combin. Theory Ser. B 92 (2004), no. 2, 325–357.

\bibitem{Whitney}
Whitney, H.,  \textit{Congruent graphs and the connectivity of graphs}. American Journal of Mathematics, 54 (1932), 61-79.
\end{thebibliography}
\end{document}